\documentclass[a4paper,american,reqno]{amsart}

\newif\ifPreprint \Preprinttrue
\newif\ifSubmission \Submissionfalse

\usepackage{babel}
\usepackage[utf8]{inputenc}
\usepackage[T1]{fontenc}
\usepackage{xspace}
\usepackage{todonotes}
\usepackage{booktabs}
\usepackage[top=4cm,bottom=4cm,right=3.8cm,left=3.8cm]{geometry}
\usepackage{csquotes}
\usepackage{microtype}
\usepackage{makecell}
\usepackage[linesnumbered,ruled,vlined]{algorithm2e}
\SetKwInput{KwInput}{Input}                %
\SetKwInput{KwOutput}{Output}              %
\usepackage[labelformat=simple]{subcaption}

\usepackage{enumitem}
\usepackage[binary-units]{siunitx}
\usepackage[gen]{eurosym}
\usepackage{pgfplots}
\usepackage[hyphens]{url}
\usepackage[hyphenbreaks]{breakurl}
\usepackage[style = authoryear,
            maxbibnames = 100,
            maxcitenames = 2,
            giveninits = true,
            uniquename = init,
            isbn = false,
            backend = bibtex]{biblatex}
\usepackage[colorlinks,
            citecolor=blue,
            urlcolor=blue,
            linkcolor=blue]{hyperref} %

\makeatletter
\patchcmd{\@settitle}{\uppercasenonmath\@title}{\scshape\large}{}{}
\patchcmd{\@setauthors}{\MakeUppercase}{\scshape\normalsize}{}{}
\makeatother

\vfuzz \hfuzz
\input{macros}

\bibliography{robust-potential-based-network-design}

\begin{document}

\title[Adjustable Robust Nonlinear Network Design]%
{Adjustable Robust Nonlinear Network Design\\
  \revTwo{Without Controllable Elements}\\
  under \rev{Load Scenario} Uncertainties}

\author[J. Thürauf, J. Grübel, M. Schmidt]%
{Johannes Thürauf, Julia Grübel, Martin Schmidt}

\address[M. Schmidt]{%
  Trier University,
  Department of Mathematics,
  Universitätsring 15,
  54296 Trier,
  Germany}
\email{martin.schmidt@uni-trier.de}

\address[J. Thürauf]{%
  University of Technology Nuremberg (UTN),
  Department Liberal Arts and Social Sciences,
  Discrete Optimization Lab,
  Dr.-Luise-Herzberg-Str. 4,
  90461 Nuremberg, Germany}
\email{johannes.thuerauf@utn.de}

\address[J. Grübel]{%
  University of Technology Nuremberg (UTN),
  Analytics and Optimization Lab,
  Dr.-Luise-Herzberg-Str. 4,
  90461 Nuremberg,
  Germany}
\email{julia.gruebel@utn.de}

\date{\today}

\begin{abstract}
  We study network design problems for nonlinear and nonconvex flow
models \revTwo{without controllable elements} under \loadscenario
uncertainties\rev{, \ie, under uncertain injections and withdrawals}.
To this end, we apply the concept of adjustable robust optimization to
compute a network design that admits a feasible transport for all,
possibly infinitely many, \load scenarios within a given uncertainty
set.
For solving the corresponding adjustable robust mixed-integer nonlinear
optimization problem, we show that a given network design is robust
feasible, i.e., it admits a feasible transport for all \loadscenario
uncertainties, if and only if a finite number of worst-case \load
scenarios can be routed through the network.
We compute these worst-case scenarios by solving polynomially many
nonlinear optimization problems.
Embedding this result for robust
feasibility in an adversarial approach leads to an exact
algorithm that computes an optimal robust network design in a
finite number of iterations.
Since all of the results are valid for general
potential-based flows, the approach can be applied to
different utility networks such as gas, hydrogen, or
water networks.
We finally demonstrate the applicability of the method by computing
robust gas networks that are protected from future \load
fluctuations.

\end{abstract}

\keywords{Robust optimization,
Nonlinear flows,
Potential-based networks,
\rev{Load scenario} uncertainties,
Mixed-integer nonlinear optimization%
}
\subjclass[2020]{90C11,
90C17,
90C35,
90C90%
}

\maketitle

\section{Introduction}
\label{sec:introduction}

Network design problems have been widely studied in the optimization
literature due to their relevance in different
applications such as transportation~\parencite{Raghunathan:2013},
telecommunication~\parencite{Koster_et_al:2013}, or supply
chains~\parencite{Santoso_et_al:2005}.
These problems typically involve optimizing a network design and the
corresponding operation so that specific \load predictions are met and the
overall costs are minimized.
In most of the cases, these models contain uncertain
parameters, which represent the deviation of the predictions from the
actual \load in the future.

In this paper, we address these uncertainties
for the class of mixed-integer nonlinear
network design problems with \loadscenario uncertainties by using
adjustable robust optimization~(\ARO).
In a nutshell, the considered adjustable robust mixed-integer
nonlinear optimization problem aims at minimizing the network
expansion costs and has the
following structure.
We first decide on the so-called
here-and-now decisions that represent
the network expansion and have to be decided before the uncertain
\loadscenario is known.
Afterward, the uncertainty realizes in a worst-case manner within an a
priori given uncertainty set.
Finally, we have to guarantee that this worst-case scenario can be
transported through the built network.
Consequently, a solution of this problem yields a robust and resilient
network, which is protected \rev{against} all, possibly infinitely many,
different \load fluctuations in the uncertainty set.
\rev{We note that these uncertain loads consist of uncertain
  injections and withdrawals in the network.}

To model the physics of the network, we use nonlinear and nonconvex
potential-based flows; see~\textcite{Gross2019}.
These flows are an extension of capacitated linear flows, which are
typically used in network design problems.
The main advantages of potential-based flows consist of their
\rev{rather} accurate representation of the underlying physics\rev{,
  which can be derived by solving the respective partial differential
  equations after certain simplifications have been imposed. In
  addition,} their broad applicability to model different types of
utility networks such as gas, hydrogen, water, or lossless DC power
flow networks \rev{is also a major advantage.}
In the following, we particularly focus on the nonlinear and nonconvex
cases.
Thus, we aim to combine mixed-integer nonlinear optimization and
robust optimization to compute resilient network designs while
accurately considering the underlying physics and taking into account
\loadscenario uncertainties.

\rev{We \revTwo{emphasize} that we focus in the following on nonconvex
  potential-based networks without controllable elements such as compressors in
  gas networks or pumps in water networks. The latter are mainly used for
  transport over long distances.
  Nevertheless, the considered models are still
  of practical relevance since \revTwo{small} distribution networks
  \revTwo{often} do not contain controllable elements and also recent
  plans for setting up hydrogen networks refrain from including
  controllable elements such as compressors in their \revTwo{start-up phase}.
  \revTwo{For example, in the Netherlands (see \textcite{LIPIAINEN_et_al:2023} or
    \textcite[Page~120]{IEA:2022}) as well as between Germany and Denmark
    \parencite{Gasunie_Energinet:2021}, pipe-only hydrogen networks are
    discussed as a first stage of building a hydrogen infrastructure.
    As mentioned by \textcite[Page~35]{Gasunie_Energinet:2021}, starting with a
    pipe-only hydrogen network has the benefit of avoiding the costly
    compressors and reduces the investment costs in the
    start-up phase.
    However, after such an initial phase it may be necessary to
    invest in compressors if the hydrogen flows increase as mentioned by
    \textcite[Page~120]{IEA:2022} for the example of the Netherlands.}}

Since the research on network design is rather extensive, we focus
on the literature about robust network design and only start with a brief
review regarding the works for
nonlinear network design without uncertainties.
Multiple approaches to solve nonlinear network design problems
are based on different relaxations of the original problem.
\textcite{Raghunathan:2013, humpola_fuegenschuh:2015} develop
different convex relaxations and embed the results in specific
branch-and-bound frameworks to solve nonlinear network design
problems.
For the case of gas networks, \textcite{Sanches_et_al:2016} develop a
mixed-integer second-order cone relaxation, which provides small gaps
\wrt\ the optimal objective value of the corresponding mixed-integer
nonlinear optimization problem~(MINLP) in many
cases.
In the recent work by \textcite{Li_et_al:2023}, the authors combine a
convex reformulation and an efficient enumeration scheme to solve a
specific gas network design problem.
For a more detailed literature review on nonlinear network design without
uncertainties, we refer to~\textcite{Li_et_al:2023} for the case of
gas networks and to~\textcite{DAmbrosio_et_al:2015} for the case of
water networks.

A large part of the literature on robust network design with uncertain
\rev{load scenarios} focuses on capacitated linear flow models.
The approaches often distinguish between two different
concepts of routing the flows.
On the one hand, there are approaches that consider a so-called static
routing.
In this case, for each uncertain \loadscenario the corresponding flows
have to follow a specific routing template, \eg, a linear
function depending on the uncertain \loadscenario.
This concept has been applied to robust network design problems with
uncertain
traffic \parencite{Koster_et_al:2013, Ben-ameur_Kerivin:2005}.
On the other hand, there are approaches using so-called dynamic
routing, in which for each uncertain \loadscenario the flows can be
chosen individually. Following this more general concept leads to an
adjustable (or
two-stage) robust mixed-integer linear network design problem;
see, \eg,~\textcite{atamtuerk_zhang:2007}.
These problems can be solved by specific branch-and-cut methods
\parencite{Cacchiani_et_al:2016} or by general methods of~\ARO;
see~\textcite{Hertog2019}.
A comparison of static and dynamic routing in addition to a so-called
affine routing is discussed in~\textcite{Poss_Rack:2013}.

We now turn to the considered case of adjustable robust network design
for nonlinear flows, which is much less researched than the case of
linear flows.
For robust gas pipeline network expansion,
\textcite{Sundar_et_al:2021} consider \rev{an interval} uncertainty
set for the demand of sinks only\rev{, \ie, they do not
consider any uncertainty regarding the injections in the
network. Further, there are no capacities for the injections, i.e.,
injections are not bounded from above.}
In this specific case, the authors show that two
worst-case scenarios suffice to guarantee robust feasibility.
\rev{However, if capacity bounds on the injections at
  nodes in the network are given, two worst-case scenarios are
  not sufficient for computing a robust network, which follows from
  our example in Section~\ref{sec:worst-case-scenarios}.}
For tree-shaped potential-based networks and a specific box
uncertainty set for the \rev{injections and withdrawals},
\textcite{Robinius2019} prove that polynomially many worst-case
\rev{load scenarios} guarantee robust feasibility.
\rev{However,} to obtain these scenarios, the authors \rev{require} the
tree structure of the network. \rev{The obtained results are applied}
to compute
a robust diameter selection for hydrogen \rev{tree-shaped} networks.
\textcite{Pfetsch_Schmitt:2023} compute robust
potential-based networks, in which no \loadscenario uncertainties are
considered, but the obtained robust network is protected from specific
arc failures.

A different notion of robustness of potential-based networks is
investigated in~\textcite{Klimm_et_al:2023}, in which network
topologies are characterized as robust if the maximal potential
differences do not increase for decreasing \rev{load scenarios}.
\rev{The authors do not consider extending the
network by additional arcs.}
For the related field of adjustable robust operation of
potential-based networks \rev{without topology design},
we refer to~\textcite{assmannNetworks2019}
as well as~\textcite{kuchlbauer_et_al:2022} and the references
therein.
Details about stochastic network design can be found in the
recent work by \textcite{bertsimas_et_al:2023}.

In this paper, we develop an exact algorithm to solve an adjustable
robust mixed-integer nonlinear network design problem with
\loadscenario uncertainties.
To this end, we focus on nonlinear and nonconvex potential-based flows
and consider general \load uncertainty sets.
Exploiting properties of potential-based flows and the underlying
network, we show that adjustable robust feasibility of a given
network expansion can be equivalently characterized by solving
polynomially many optimization problems.
These optimization problems consist of maximizing, respectively
minimizing, specific network characteristics such as arc
flows or potential differences \wrt\ the uncertainty set.
Solving the latter problems leads to a finite set of
worst-case \load scenarios, which prove adjustable robust feasibility
or infeasibility of the considered network expansion.
Embedding this characterization in an
exact adversarial approach leads to an algorithm that solves the
considered adjustable robust mixed-integer nonlinear optimization
problem in a finite number of iterations.
The algorithm starts with a small subset of \load
scenarios that is iteratively augmented by worst-case \load scenarios
obtained by the developed characterization of robust feasibility.
We finally demonstrate the applicability of the developed approach by
computing adjustable robust gas networks that are protected from
future \load fluctuations.
The numerical results show that only a small number of worst-case
scenarios suffices to obtain an adjustable robust network design in
practice.
\rev{Consequently, our main contributions can be summarized as follows.
  \begin{itemize}
  \item[(i)] We characterize robust feasibility of a given
    potential-based network with
    uncertain injections and withdrawals by finitely many
    worst-case scenarios.
    In particular, the network topology can be arbitrary,
    the uncertainty can be both in the injections and withdrawals,
    and the potential-based flow model can be nonconvex.
    Moreover, the only assumptions regarding the uncertainty sets
    that we need to impose are non-emptiness and compactness.
    To the best of our knowledge, this very general setup has not been
    considered so far in the literature.
  \item[(ii)] We derive how to compute these worst-case
    scenarios by solving a polynomial number of nonlinear
    optimization problems.
  \item[(iii)] We use these results to design an adversarial
    approach that is an exact algorithm for solving the robust
    mixed-integer nonlinear network design problem to global
    optimality.
  \item[(iv)] We demonstrate the applicability of our approach by
    applying the developed algorithm together with enhanced solution
    techniques to compute robust pipe-only gas networks of
    realistic size.
  \end{itemize}}

The remainder of the paper is organized as follows.
In Section~\ref{sec:problem-statement}, we introduce
potential-based flows and state the considered adjustable robust
mixed-integer nonlinear network design problem under \loadscenario
uncertainties.
In Section~\ref{sec:exact-adversarial-approach}, we derive an
characterization of adjustable robust feasibility of a
given network expansion based on finitely many worst-case \load
scenarios.
Subsequently, we embed this result in an
exact adversarial approach that solves the uncertain network design
problem.
We present different solution techniques that speed up the
performance of the developed approach in
Section~\ref{sec:enhanced-solution-techniques}.
Using an academic example, we then discuss that the number of
necessary worst-case \load scenarios in the algorithm can
significantly vary depending on the capacity of the sources; see
Section~\ref{sec:worst-case-scenarios}.
We finally demonstrate the applicability of the developed approach using
the example of gas networks in Section~\ref{sec:computational-study},
followed by a discussion of possible future research directions in
Section~\ref{sec:conclusion}.

\section{Problem Statement}
\label{sec:problem-statement}

We now introduce the considered nonlinear potential-based flow model in
Section~\ref{subsec:potential-flows} before we state the adjustable
mixed-integer nonlinear network design problem in
Section~\ref{subsec:robust-design-model}.

\subsection{Potential-Based Networks}
\label{subsec:potential-flows}

We consider potential-based flows to model the
underlying physical laws of the network flow.
Potential-based flows form an extension of classic linear
capacitated flow models and we now formally introduce them based
on~\textcite{Gross2019, Labbe2019}.
Let $\graph=(\nodes,\arcs)$ be a directed
multi-graph consisting of
a set of nodes~$\nodes$ and a set of arcs~$\arcs$.
The set of nodes~$\nodes$ is partitioned into nodes~$\entries$ at
which flow is injected, nodes~$\exits$ at which flow is withdrawn, and
inner nodes~$\innodes$ at which neither flow is injected nor withdrawn.
Furthermore, the set~$\arcs$ represents the arcs of the network and
consists of triples~$(u,v,\labelArc)$.
Here, $u$ and $v$ represent the start and end node of the arc~$\arc$
and~$\labelArc$ is the label of the arc. \rev{In the following, we
  denote by $L$ the set of all possible labels.
  Consequently, for every triple~$(u,v,\labelArc)$, it
  holds~$(u,v,\labelArc) \in V \times V \times L$.
  With a slight abuse of notation we also call such triples an arc and
  write $(u,v,\labelArc) \in A$ as well.}
This modeling choice allows to consider multiple parallel arcs between
two nodes, which often \rev{occur} in real-word utility networks.

In addition to the flow variables~$\massflow \in
\reals^{\arcs}$, we consider nodal
potential levels~\mbox{$\potential \in \reals^{\nodes}$}.
Due to technical restrictions, both the flow and the
potential variables are bounded, \ie,
\begin{equation*}
  \lbPot_{\node} \leq \potential_{\node} \leq \ubPot_{\node},
  \quad \node \in \nodes, \quad
  \lbMassflow_{\arc} \leq \massflow_{\arc} \leq \ubMassflow_{\arc},
  \quad \arc \in \arcs.
\end{equation*}
To model the case of unbounded potentials
or uncapacitated flows, we can set the potential bounds~$\lbPot_\node
\leq \ubPot_\node, \node \in \nodes,$ and the arc flow
bounds~\mbox{$\lbMassflow_\arc \leq \ubMassflow_\arc, \arc \in
  \arcs,$}
to $\pm \infty$.

For a given arc~$\arc \in \arcs$, the incident potentials and
the corresponding arc flow are coupled by a \socalled potential
function~$\potFunc_{\arc} \colon \reals \rightarrow \reals$.
The potential function is usually nonlinear and nonconvex.
We further assume that the properties
\begin{enumerate}[label=(\roman*)]
\item
  $\potFunc_{\arc}$ is continuous,
\item
  $\potFunc_{\arc}$ is strictly increasing, and
\item
  $\potFunc_{\arc}$ is odd, \ie, $\potFunc_{\arc}(-x) =
  -\potFunc_{\arc}(x)$,
\end{enumerate}
hold, which are
natural in the context of utility networks.
The coupling between potentials and arc flows is given by
\begin{equation*}
  \potential_{\node} - \potential_{\otherNode} =
  \potFunc_{\arc}(\massflow_{\arc}), \quad \arc=(\node,\otherNode,
  \labelArc) \in \arcs.
\end{equation*}

\rev{We further consider a \loadscenario~$\loadFlowVec \in
\reals^{\nodes}$. Injections at node $\node \in \nodes$ are represented by 
$\loadFlowVec_{\node} < 0$, while $\loadFlowVec_{\node} > 0$ indicates 
withdrawals at node $\node \in \nodes$. For each inner 
node~\mbox{$\node \in \innodes$}, \mbox{$\loadFlowVec_{\node} = 0$} 
holds.}
Since we consider stationary flows, this \loadscenario~$\loadFlowVec
\in
\reals^{\nodes}$ has to be balanced, \ie, the total amount of
injections equals the total amount of withdrawals:
$\sum_{\node \in \entries} \loadFlowVec_{\node}
\rev{+} \sum_{\node \in \exits} \loadFlowVec_{\node} \rev{=0}$.
We further have to impose mass flow conservation by
\begin{equation*}
  \sum_{\arc \in \rev{\inArcs[\node]}} \massflow_{\arc}
  - \sum_{\arc \in \rev{\outArcs[\node]}} \massflow_{\arc} =
  \rev{d_{\node}, \quad \node \in \nodes.}
\end{equation*}

Combining the previous constraints leads to the formal definition of a
potential-based flow.

\begin{definition}
  For a given \loadscenario~$\loadFlowVec \in \reals^{\nodes}$
  with~$\loadFlowVec_{\node} = 0$ for all $\node \in \innodes$,
  a tuple~$(\massflow, \potential)$ is a feasible
  potential-based flow if and only if it satisfies
  \begin{align*}
  \sum_{\arc \in \rev{\inArcs[\node]}} \massflow_{\arc}
  &- \sum_{\arc \in \rev{\outArcs[\node]}} \massflow_{\arc} =
  \rev{d_{\node}, \quad \node \in \nodes.}
    \\
    \potential_{\node} - \potential_{\otherNode}
    &= \potFunc_{\arc}(\massflow_{\arc}), \quad
      \arc=(\node,\otherNode, \labelArc) \in
      \arcs,
    \\
    \lbPot_{\node} \leq \potential_{\node}
    &\leq \ubPot_{\node}, \quad \node \in \nodes, \\
    \lbMassflow_{\arc} \leq \massflow_{\arc}
    & \leq \ubMassflow_{\arc}, \quad \arc \in \arcs.
  \end{align*}
\end{definition}

One of the main advantages of using potential-based flows lies in
their strong modeling capabilities \wrt\ flows in  utility networks.
In~\textcite{Gross2019}, explicit potential functions for stationary
gas ($\potFunc^{\text{G}}$), water ($\potFunc^{\text{W}}$), and
lossless DC power-flow networks ($\potFunc^{\text{DC}}$) are presented.
For an arc~$\arc \in \arcs$ and a corresponding arc
flow~$\massflow_{\arc}$, these potential functions are explicitly
given by
\begin{equation}
  \label{different-potential-functions}
  \potFunc^{\text{G}}(\massflow_{\arc}) = \Lambda_{\arc}
  \massflow_{\arc} \abs{\massflow_{\arc}},
  \quad
  \potFunc^{\text{W}}(\massflow_{\arc}) = \Lambda_{\arc}
  \text{sgn}(\massflow_{\arc}) \abs{\massflow_{\arc}}^{1.852},
  \quad
  \potFunc^{\text{DC}}(\massflow_{\arc}) = \Lambda_{\arc}
  \massflow_{\arc},
\end{equation}
where~$\Lambda_{\arc} > 0$ is an arc specific constant depending on
the application.

\subsection{Robust Network Design}
\label{subsec:robust-design-model}

We now present an adjustable robust network expansion model that takes
\loadscenario uncertainties into account.
For modeling the underlying physics of the network flows, we use the
potential-based flows as previously introduced.

In general, \load forecasts that are considered in the
network design process are affected by uncertainties.
Taking these \loadscenario uncertainties into account is of high
relevance since even small perturbations of
the injections and withdrawals can render the planned network design
infeasible, \ie, the \loadscenario cannot be transported through the
network.
We now address these \loadscenario uncertainties by applying the
well-established concept of (adjustable) robust optimization; see,
\textcite{Hertog2019} and the references therein.
To this end, we consider the uncertainty set
\begin{equation}
  \label{eq:demand-uncertainty-set}
  \uncertaintySet \define Z
  \cap
  \Defset{\loadFlowVec \in \reals^{\nodes}}{\sum_{\node \in \entries}
    \loadFlowVec_{\node} \rev{+} \sum_{\node
      \in \exits}   \loadFlowVec_{\node} \rev{=0},
    \ \loadFlowVec_{\node} = 0, \
    \node \in \innodes},
\end{equation}
of balanced \rev{load scenarios}.
\rev{Here, $Z \subset \reals^{\nodes}$ is any non-empty and compact
  set that describes the uncertainty in the injections and
  withdrawals at the nodes of the network.
  Then, we ensure by~\eqref{eq:demand-uncertainty-set} 
  that the uncertainty set~$U$ contains all balanced load scenarios
  of~$Z$, \ie, scenarios in~$Z$ for which the total amount of
  injections equals the total amounts of withdrawals.
  The latter is necessary because an
  unbalanced load scenario immediately renders every stationary flow model
  infeasible.}
We note that \rev{definition \eqref{eq:demand-uncertainty-set} allows
  to consider} convex, nonconvex, or even discrete uncertainty sets.

With this uncertainty set at hand, the task of computing an adjustable
robust network design consists of finding a cost-optimal network
design such that for each \loadscenario~$\loadFlowVec \in U$,
there is a feasible transport through the built network.

For stating a corresponding adjustable robust optimization model, we
partition the set of arcs~$\arcs$ into existing arcs~$\exArcs$
and into candidate arcs~$\expArcs$ that can be built to enhance the
capacity of the network.
This allows to design a network from scratch ($\exArcs =
\emptyset$) as well as to increase the capacity of existing networks
($\exArcs \neq \emptyset$).
We further introduce binary variables $\expVar \in X \subseteq
\set{0,1}^{\expArcs}$.
Here, for an arc~$\arc \in \expArcs$, the
binary variable~$\expVar_{\arc}$ equals one if the candidate
arc~$\arc$ is built and otherwise, it is zero.
Further, expanding the network by an arc~$\arc \in \expArcs$
results in costs~$c_{\arc} > 0$.
The set~$X$ can contain additional constraints on the
expansion decisions such as that only one out of multiple parallel
arcs in $\expArcs$ can be built.
The latter occurs in the discrete selection of pipeline diameters in
gas networks; see, e.g., \textcite{Li_et_al:2023}.
We are now ready to state a model that computes an adjustable robust
network design that guarantees that each \loadscenario in the
uncertainty set~$\uncertaintySet$ can be transported:
\begin{subequations} \label{expansion}
  \begin{align}
    \minCostsRobustExpansion(\uncertaintySet) \define \min_{\expVar,
    \massflow, \potential} \quad
    & \sum_{\arc \in \expArcs} c_{\arc}
      \expVar_{\arc} \label{eq:expansion:objective-function} \\
    \st \quad
    & \expVar \in X,
      \label{eq:expansion:expansion-constraints}
    \\
    & \forall \loadFlowVec \in \uncertaintySet \, \exists
      \left(\massflow^{\loadFlowVec},
      \potential^{\loadFlowVec}\right) \text{ with
      } \label{eq:expansion:quantifiers-robust} \\
    & \quad
      \sum_{\arc \in \rev{\inArcs[\node]}}
      \massflow^{\loadFlowVec}_{\arc} - \sum_{\arc \in
      \rev{\outArcs[\node]}} \massflow^{\loadFlowVec}_{\arc} =
      \rev{\loadFlowVec_{\node}, \quad \node \in \nodes,}
                              \label{eq:expansion:massflow-conservation}
    \\
    & \quad
      \potential^{\loadFlowVec}_{\node} - \potential^{\loadFlowVec}_{\otherNode} =
      \potFunc_{\arc}(\massflow^{\loadFlowVec}_{\arc}), \quad \arc=(\node,\otherNode, \labelArc)
      \in \exArcs,
      \label{eq:expansion:potential-coupling-existingArcs}
    \\
    & \quad
      \potential^{\loadFlowVec}_{\node} - \potential^{\loadFlowVec}_{\otherNode} \leq
      \potFunc_{\arc}(\massflow^{\loadFlowVec}_{\arc}) +
      (1-\expVar_{\arc}) M^{+}_{\arc},
      \quad \arc=(\node,\otherNode, \labelArc)
      \in \expArcs,
      \label{ineq:expansion:potential-coupling-candidateArcs-uBound}
    \\
    & \quad
      \potential^{\loadFlowVec}_{\node} - \potential^{\loadFlowVec}_{\otherNode} \geq
      \potFunc_{\arc}(\massflow^{\loadFlowVec}_{\arc}) +
      (1-\expVar_{\arc}) M^{-}_{\arc},
      \quad \arc=(\node,\otherNode, \labelArc)
      \in \expArcs,
      \label{ineq:expansion:potential-coupling-candidateArcs-lBound}
    \\
    & \quad
      \lbMassflow_{\arc} \leq \massflow^{\loadFlowVec}_{\arc} \leq \ubMassflow_{\arc},
      \quad \arc \in \exArcs,
      \label{eq:expansion:massflow-bounds-exArcs}
    \\
    & \quad \lbMassflow_{\arc} \expVar_{\arc} \leq \massflow^{\loadFlowVec}_{\arc} \leq
      \ubMassflow_{\arc} \expVar_{\arc},
      \quad \arc \in \expArcs,
      \label{eq:expansion:massflow-bounds-candidateArcs}
    \\
    & \quad  \lbPot_{\node} \leq \potential^{\loadFlowVec}_{\node} \leq \ubPot_{\node},
      \quad \node \in \nodes.
      \label{eq:expansion:potential-bounds}
  \end{align}
\end{subequations}

In the objective function~(\ref{eq:expansion:objective-function}), we
minimize the costs associated to the chosen network design.
Constraints~(\ref{eq:expansion:expansion-constraints}) impose
additional restrictions on the network design.
Then, for every \loadscenario~$\loadFlowVec \in \uncertaintySet$,
we determine flows~$\massflow^{\loadFlowVec}$ and
potentials~$\potential^{\loadFlowVec}$ such that
mass flow conservation~\eqref{eq:expansion:massflow-conservation} and
the
potential bounds~\eqref{eq:expansion:potential-bounds} are satisfied.
Furthermore, we ensure by
\mbox{Constraints~\eqref{eq:expansion:potential-coupling-existingArcs}--\eqref{ineq:expansion:potential-coupling-candidateArcs-lBound}}
that the coupling between the potentials and the arc flows is
satisfied for all existing arcs and all candidate arcs that are built.
Moreover, we guarantee that for existing and new arcs
specific flow capacities are satisfied. If a candidate arc~$\arc
\in \expArcs$ is
not built, \ie, $\expVar_{\arc}=0$, then the corresponding arc flow is
set to zero; see
Constraints~\eqref{eq:expansion:massflow-bounds-candidateArcs}.

In line with~\textcite{Schmidt_Thuerauf:2022}, we now discuss that for
each arc~$(u,v, \labelArc)$, the big-$M$ values
\begin{equation} \label{eq:big-Ms-expansion-model}
  M^{+}_{\arc} = \ubPot_u - \lbPot_v, \quad
  M^{-}_{\arc} = \lbPot_u - \ubPot_v,
\end{equation}
are valid. Here, valid means that if a candidate arc~$\arc=(u,v, \labelArc) \in
\expArcs$ is not built, \ie, $x_{\arc} = 0$, the
Constraints~\eqref{ineq:expansion:potential-coupling-candidateArcs-uBound}
and~\eqref{ineq:expansion:potential-coupling-candidateArcs-lBound}
are redundant and we have no coupling between the incident potentials
and the corresponding arc flow.
To see this, let us consider a candidate arc~$\arc \in \expArcs$ with
$x_{\arc} = 0$ and an arbitrary uncertain \loadscenario~$\loadFlowVec \in
\uncertaintySet$.
Then, from
Constraints~\eqref{eq:expansion:massflow-bounds-candidateArcs}, it
follows~$\massflow^{\loadFlowVec}_{\arc} = 0$.
Consequently, from Constraints~\eqref{eq:expansion:potential-bounds}
and $\potFunc_{\arc}(0) = 0$,
we obtain $M^{-}_{\arc} = \lbPot_u - \ubPot_v \leq \potential^{\loadFlowVec}_{u} -
\potential^{\loadFlowVec}_{v} \leq \ubPot_u - \lbPot_v = M^{+}_{\arc}$, which
corresponds to
Constraints~\eqref{ineq:expansion:potential-coupling-candidateArcs-uBound}
and~\eqref{ineq:expansion:potential-coupling-candidateArcs-lBound}.

From the perspective of robust optimization, Problem~\eqref{expansion} is
an adjustable robust optimization problem.
The expansion variables~$\expVar$ represent the first-stage or
\socalled ``here-and-now'' decisions and the flows~$\massflow^{\loadFlowVec}$ as well
as the potentials~$\potential^{\loadFlowVec}$ are second-stage or \socalled
``wait-and-see'' decisions, which are adapted for each
uncertainty~$\loadFlowVec \in \uncertaintySet$.
\begin{remark}
  For linear potential functions~$\potFunc_{\arc}$,
  Problem~\eqref{expansion} is an adjustable robust mixed-integer linear
  optimization problem that can be tackled by standard methods of
  adjustable linear robust optimization, \eg, by methods based on
  column-and-constraint generation; see~\textcite{Zhao2012AnEA,
    Lefebvre:2022}.
\end{remark}
In the light of this remark, we now focus on nonlinear and nonconvex
potential functions~$\potFunc_{\arc}$, which occur, \eg, in gas,
hydrogen, or water networks; see~\eqref{different-potential-functions}.
Thus, we obtain an adjustable robust mixed-integer nonlinear optimization
problem~\eqref{expansion}, for which the set of applicable methods of
the literature is scarce.

\section{Exact Adversarial Approach}
\label{sec:exact-adversarial-approach}

We now follow the idea of the well-known adversarial approach in
robust optimization~\parencite{Bienstock2008}, to solve
Problem~\eqref{expansion} to global optimality.
The main idea of the adversarial approach is to replace the
original uncertainty set~$U$ by a finite set~$S \subseteq U$ of
``worst-case'' scenarios.
To this end, the approach starts with a small set of scenarios~$S$ and
then solves Problem~\eqref{expansion} \wrt\ $S$ instead of $U$.
The latter problem is now a mixed-integer nonlinear optimization problem
consisting of finitely many variables and constraints due to $\abs{S}
< \infty$.
If the obtained solution is robust feasible, \ie, it is feasible for
the original problem \wrt~$U$, then it
is also optimal due to $S \subseteq U$.
Otherwise, there are scenarios~$U \setminus S$ that render the obtained
point infeasible.
If this is the case, at least one of these scenarios is added to~$S$
and the procedure repeats by solving the robust problem \wrt\ the
updated $S$.
When applying the adversarial approach, the most important questions
to answer are:
\begin{enumerate}[label=(\roman*)]
  \item \label{adversarial:robust-feasibility}
    How to verify that a given point is robust feasible?
  \item \label{adversarial:computing-infeasible-scn}
    How to compute a scenario in~$U \setminus S$ that
    certifies the infeasibility of a given point?
  \item \label{adversarial:finite-termination}
    Does the adversarial approach terminate after a finite
    number of steps?
\end{enumerate}
For strict robust optimization,
Questions~\ref{adversarial:robust-feasibility}
and~\ref{adversarial:computing-infeasible-scn} are usually addressed
by maximizing the constraint violation \wrt\ the uncertainty
set and a fixed ``here-and-now'' decision.
Furthermore, for linear constraints and polyhedral uncertainty sets,
the method terminates after a finite number of steps;
see~\textcite{Bertsimas2016}.

However, for the considered case of~\ARO,
applying the adversarial approach is even more challenging
since we cannot directly transfer the idea to
compute a violating scenario of the strictly robust case to the
adjustable robust one.
This is mainly based on the fact that in~\ARO, we can determine the
second-stage decisions after the
uncertainty realizes.
For adjustable robust linear problems with polyhedral uncertainty sets,
it is shown that adding finitely many ``worst-case'' scenarios
suffices; see~\textcite{ayoub2016}.
An analogous result is shown for adjustable robust nonconvex optimization
with uncertainty sets being polytopes under specific quasi-convexity
assumptions; see~\textcite{Takeda2008}.
However, for the considered adjustable robust mixed-integer nonlinear
problem and the general choice of uncertainty
set~(\ref{eq:demand-uncertainty-set}), these approaches cannot be
directly applied.

We now exploit structural properties of potential-based flows and the
underlying graph to answer
Questions~\ref{adversarial:robust-feasibility}--\ref{adversarial:finite-termination}
for the considered Problem~(\ref{expansion}).
In particular, we show that for given first-stage decisions~$x \in
X$, we can verify robust feasibility, respectively compute a
violating scenario, by solving polynomially (in the encoding-length of
the underlying graph) many single-level nonlinear optimization
problems.
To this end, we start with the case of weakly connected graphs and
then extend these results to general graphs.
For a given expansion decision~$x \in X$, we now consider three
different classes of nonlinear optimization
problems.
Solving these ``adversarial'' problems either verifies robust
feasibility of~$x$ or yields a \loadscenario~$\loadFlowVec \in U$
that certifies the infeasibility of~$x$.

\rev{
  Note that for a given expansion decision~$x \in X$, we can
  directly remove every arc~$\arc$ that is not built, i.e., every arc
  for which $x_{\arc} = 0$ holds, because the corresponding
  constraints of~(\ref{ineq:expansion:potential-coupling-candidateArcs-uBound})
  and of~(\ref{ineq:expansion:potential-coupling-candidateArcs-lBound})
  are redundant. Further, the corresponding flow variable~$\massflow_{\arc}$
  is zero due to~(\ref{eq:expansion:massflow-bounds-candidateArcs}).
  Consequently, we can directly remove these constraints and the
  corresponding flow variable~$\massflow_{\arc}$
  in~\mbox{(\ref{eq:expansion:massflow-conservation})--(\ref{eq:expansion:massflow-bounds-candidateArcs})},
  which we implicitly assume for the remaining section when considering a
  network with a fixed expansion decision~$x$ such as in the adversarial
  problems~(\ref{eq:maximum-potential-difference})--(\ref{eq:maximum-arc-flow})
  below.}

First, for a given pair of nodes~$(u, v) \in \nodes^{2}$ \rev{with $\node
\neq \otherNode$}, we compute the
maximum potential difference between~$\node$
and~$\otherNode$ within the uncertainty set~$\uncertaintySet$ by
\begin{equation}
  \label{eq:maximum-potential-difference}
  \maxPotDiff_{\node,\otherNode}(\expVar) \define \max_{\loadFlowVec,
  \massflow, \potential} \ \potential_{\node} -
  \potential_{\otherNode}
  \quad
  \st \quad
  \eqref{eq:expansion:massflow-conservation}\text{--}%
  \eqref{ineq:expansion:potential-coupling-candidateArcs-lBound},
  \
  \loadFlowVec \in \uncertaintySet.
\end{equation}
In Problem~\eqref{eq:maximum-potential-difference}, we explicitly
dismiss the flow and potential
bounds~\eqref{eq:expansion:massflow-bounds-exArcs}--\eqref{eq:expansion:potential-bounds}.
The intuition behind this is to compute scenarios
that induce the most stress on the network \wrt\ the potential
levels, \ie, we are particularly interested in scenarios that
violate the potential bounds.
For the given expansion decision~$x$, we will later show that we can
only find feasible second-stage decisions~$\potential$ if
and only if the objective value of
Problem~\eqref{eq:maximum-potential-difference} stays below
specific bounds.
Thus, solving Problem~\eqref{eq:maximum-potential-difference} for
each pair of nodes leads to finitely many ``worst-case'' scenarios
regarding the potential levels~$\potential$ and the given expansion
decision~$x$.
We note that these worst-case scenarios are also considered in, \eg,
\textcite{Labbe2019, assmannNetworks2019, Robinius2019}
in the context of gas market problems, of robust control of gas
networks, and of robust selection of diameters in tree-shaped
networks.
Moreover, Problem~(\ref{eq:maximum-potential-difference}) can
be \rev{solved} in polynomial time for box uncertainty sets and tree-shaped
networks~\parencite{Robinius2019}.
However, it is NP-hard for general potential-based flows in
general graphs; see~\textcite{Thuerauf2022}.

Second and third, we compute for each arc~$\arc \in \arcs$ the minimum and
maximum arc flow within the considered uncertainty
set~$\uncertaintySet$ by
\begin{equation}
  \label{eq:minimum-arc-flow}
  \minArcFlow_{\arc}(\expVar) \define
  \min_{\loadFlowVec, \massflow, \potential} \
  \massflow_{\arc} \quad \st \quad
  \eqref{eq:expansion:massflow-conservation}\text{--}%
  \eqref{ineq:expansion:potential-coupling-candidateArcs-lBound},
  \
  \loadFlowVec \in \uncertaintySet
\end{equation}
and
\begin{equation}
  \label{eq:maximum-arc-flow}
  \maxArcFlow_{\arc}(\expVar) \define
  \max_{\loadFlowVec, \massflow, \potential} \
  \massflow_{\arc} \quad \st \quad
  \eqref{eq:expansion:massflow-conservation}\text{--}%
  \eqref{ineq:expansion:potential-coupling-candidateArcs-lBound},
  \
  \loadFlowVec \in \uncertaintySet.
\end{equation}
We again dismiss potential and flow bounds in
Problems~\eqref{eq:minimum-arc-flow} and~\eqref{eq:maximum-arc-flow}
because we are particularly interested in finding scenarios that
violate the flow bounds.
Analogously, we will show that for the given expansion decision~$x$, we can
only find feasible second-stage decisions~$\massflow$ if and only if
the objective values of (\ref{eq:minimum-arc-flow})
and~(\ref{eq:maximum-arc-flow}) satisfy specific bounds.
Thus, solving these problems leads to a finite set of ``worst-case''
scenarios regarding the flows.

We now prove that we can verify robust feasibility of given
first-stage decisions~\mbox{$x \in X$} by solving the polynomially
many
Problems~(\ref{eq:maximum-potential-difference})--(\ref{eq:maximum-arc-flow}).
\rev{To this end, we introduce the notation }
\begin{equation*}
  \rev{  \arcs(x) \define \exArcs \cup \defset{\arc \in
    \expArcs}{\expVar_{\arc} = 1},
  \quad
  \bar{\arcs}(x) \define \defset{\arc \in \expArcs}{\expVar_{\arc} =0}.
}
\end{equation*}
\rev{Here, $\arcs(x)$ is the set of existing and built arcs
  and $\bar{\arcs}(x)$ is the set of expansion arcs that are not built.
Consequently, $G(x)=(\nodes, \rev{\arcs(x)})$ describes the expanded graph
that is realized according to the given expansion decision~$x$.}
\rev{In the following, }we use the auxiliary lemma from the literature,
which states that for a given \loadscenario the corresponding flows and
potential differences are unique.
\begin{lemma} \label{lemma:uniquness-of-flows-passive-case}
  Let $\expVar \in X$ be fixed and
  let~$\graph'(\expVar)=(\nodes, \rev{\arcs(x)})$ be the expanded graph.
  Further, we assume that $\graph'(\expVar)$ is weakly connected.
  For a fixed \loadscenario~$\loadFlowVec \in \uncertaintySet$,
  there are potentials~$\potential'$ and unique flows~$\massflow$ such
  that the set of feasible points that satisfies
  Constraints~\eqref{eq:expansion:massflow-conservation}--%
  \eqref{ineq:expansion:potential-coupling-candidateArcs-lBound}
  and $\massflow_{\arc} = 0$ for each arc \rev{not built, \ie, $\arc
    \in \bar{\arcs}(x),$} is non-empty and given by
  \begin{equation*}
    \defset{(q, \potential)}{\potential = \potential' + \ones \eta, \
      \eta \in \reals}
  \end{equation*}
  where $\ones$ is a vector of ones in appropriate dimension.
\end{lemma}
\begin{proof}
  The lemma follows from Theorem~7.1 of~\textcite{humpola7} and the
  fact that for every arc~$\arc$ with $\expVar_{\arc} = 0$, the
  corresponding flows are zero.
\end{proof}

We now characterize robust feasibility of given expansion decisions~$x
\in X$
using
Problems~\eqref{eq:maximum-potential-difference}--\eqref{eq:maximum-arc-flow}
for the case of a weakly connected expanded graph.

\begin{theorem}
  \label{theorem:characterization-single-connected-component}
  Let $\expVar \in X$ be fixed and
  let~$\graph'(\expVar)=(\nodes, \rev{\arcs(x)})$ be the expanded
  graph.
  Further, we assume that $\graph'(\expVar)$ is weakly~connected.
  Then, Constraints~\eqref{eq:expansion:quantifiers-robust}--%
  \eqref{eq:expansion:potential-bounds}
  are satisfied \wrt~$\expVar$ if and only if for every pair of
  nodes~$(\node, \otherNode) \in \nodes^{2}$ \rev{ with $\node \neq
    \otherNode$}, the corresponding maximum
  potential difference satisfies the potential bounds
  \begin{align}
    \label{ineq:maxPot-satisfies-bounds}
    \maxPotDiff_{\node, \otherNode}(\expVar)
    \leq \ubPot_{\node} - \lbPot_{\otherNode}
  \end{align}
  and for each arc~$\arc \in \arcs'(\expVar)$, the minimum and maximum
  arc flow satisfies the corresponding flow bounds, \ie,
  \begin{equation} \label{ineq:minimum-maximum-flow-satisfies-bounds}
    \minArcFlow_\arc(\expVar) \geq \lbMassflow_\arc
    \quad \text{ and } \quad
    \maxArcFlow_{\arc}(\expVar) \leq \ubMassflow_{\arc}.
  \end{equation}
\end{theorem}
\begin{proof}
  For a fixed expansion~$\expVar \in X$, let
  Constraints~\eqref{eq:expansion:quantifiers-robust}--%
  \eqref{eq:expansion:potential-bounds} be satisfied.
  We now distinguish two cases.
  First, we assume for the sake of contradiction that an arc~$\arc
  \in \arcs(\expVar)$ exists such that $\minArcFlow_{\arc}(\expVar) <
  \lbMassflow_{\arc}$ holds.
  Let~$(\loadFlowVec, \massflow, \potential)$ be a corresponding
  optimal solution of~\eqref{eq:minimum-arc-flow}.
  Applying
  Lemma~\ref{lemma:uniquness-of-flows-passive-case} to
  \loadscenario~$\loadFlowVec$ shows that there
  are unique flows~$\massflow$ satisfying
  Constraints~\eqref{eq:expansion:massflow-conservation}--%
  \eqref{ineq:expansion:potential-coupling-candidateArcs-lBound}
  and $\massflow_{\arc} = 0$ for each arc \rev{that is not built, \ie,
  $\arc \in \bar{\arcs}(x)$}.
  Due to the feasibility of
  Constraints~\eqref{eq:expansion:quantifiers-robust}--%
  \eqref{eq:expansion:potential-bounds}
  \wrt~$\expVar$, these flows satisfy
  Constraints~\eqref{eq:expansion:massflow-bounds-exArcs}
  and~\eqref{eq:expansion:massflow-bounds-candidateArcs}.
  This contradicts the assumption~$\minArcFlow_{\arc}(\expVar) <
  \lbMassflow_{\arc}$.
  Thus, $\minArcFlow_\arc(\expVar) \geq \lbMassflow_\arc$ is true
  for each $\arc \in \arcs(\expVar)$.
  The case of the upper flow bound can be handled analogously.
  Consequently,
  Conditions~\eqref{ineq:minimum-maximum-flow-satisfies-bounds} hold.

  Second, we now assume for the sake of contradiction that there is a
  pair of nodes~\mbox{$(\node, \otherNode) \in \nodes^{2}$} such that
  $\maxPotDiff_{\node,\otherNode}(\expVar) > \ubPot_{\node} -
  \lbPot_{\otherNode}$.
  Let~$(\loadFlowVec, \massflow, \potential)$ be a corresponding
  optimal solution of~\eqref{eq:maximum-potential-difference}.
  \rev{Due to the feasibility of
  Constraints~\eqref{eq:expansion:quantifiers-robust}--%
  \eqref{eq:expansion:potential-bounds} \wrt~$\expVar$, for load
  scenario~$\loadFlowVec$ there exist
  flows and potentials satisfying these constraints.
  Moreover, the corresponding flows are unique due to Lemma~1.
  Consequently, there is a point~$(\loadFlowVec, \massflow,
  \potential')$ that is feasible for
  Constraints~\eqref{eq:expansion:massflow-conservation}--%
  \eqref{eq:expansion:potential-bounds}
  and that satisfies $\massflow_{\arc} = 0$ for each arc \rev{not
    built, \ie, $\arc \in \bar{\arcs}(x)$}.
}
  The potential bounds~\eqref{eq:expansion:potential-bounds}
  imply~$\potential'_{\node} - \potential'_{\otherNode} \leq
  \ubPot_{\node} - \lbPot_{\otherNode}.$
  From Lemma~\ref{lemma:uniquness-of-flows-passive-case}, it follows
  that there is an $\eta \in \reals$ so that $\potential' + \ones\eta = \potential$ holds.
  Consequently, we obtain the contradiction
  \begin{equation*}
    \ubPot_{\node} - \lbPot_{\otherNode} \geq \potential'_{\node} -
    \potential'_{\otherNode} =
    \potential'_{\node} + \eta - (\potential'_{\otherNode} + \eta) =
    \potential_{\node} - \potential_{\otherNode}
    = \maxPotDiff_{\node,\otherNode}(\expVar).
  \end{equation*}

  We now examine the reverse direction.
  Thus, for fixed
  expansion~$\expVar \in X$,
  Conditions~\eqref{ineq:maxPot-satisfies-bounds}
  and~\eqref{ineq:minimum-maximum-flow-satisfies-bounds} are
  satisfied.
  Let~$\loadFlowVec \in \uncertaintySet$ be an arbitrary
  \loadscenario.
  Due to Lemma~\ref{lemma:uniquness-of-flows-passive-case},
  there is a feasible point~$(\loadFlowVec, \massflow, \potential)$
  that satisfies
  Constraints~\eqref{eq:expansion:massflow-conservation}--%
  \eqref{ineq:expansion:potential-coupling-candidateArcs-lBound}
  and $\massflow_{\arc} = 0$ for each arc \rev{not built, \ie, $\arc
    \in \bar{\arcs}(x)$}.
  In addition, Lemma~\ref{lemma:uniquness-of-flows-passive-case}
  implies that we can shift the potentials so that
  $\potential_{\node} \leq \ubPot_{\node}$ for every node~$\node \in
  \nodes$ holds and there is a
  node~$w$ with $\potential_{w} = \ubPot_{w}$.
  This point~$(\loadFlowVec, \massflow, \potential)$ is feasible for
  Problems~\eqref{eq:minimum-arc-flow}
  and~\eqref{eq:maximum-arc-flow} since no arc flow or potential
  bounds are present in these problems.
  Consequently, from
  Condition~\eqref{ineq:minimum-maximum-flow-satisfies-bounds}, it
  follows that the flow
  bounds~\eqref{eq:expansion:massflow-bounds-exArcs}
  and~\eqref{eq:expansion:massflow-bounds-candidateArcs} are
  satisfied.

  We now assume for the sake of contradiction that there is a node~$h
  \in \nodes$ with $\potential_{h} < \lbPot_{h}$.
  Then, it follows
  \begin{equation*}
    \potential_{w} - \potential_{h} = \ubPot_{w} - \potential_{h} >
    \ubPot_{w} - \lbPot_{h},
  \end{equation*}
  which is a contradiction to
  Condition~\eqref{ineq:maxPot-satisfies-bounds} since $(\loadFlowVec,
  \massflow, \potential)$ is a feasible point for
  Problem~\eqref{eq:maximum-potential-difference} \wrt\ the
  pair of nodes~$(w,h)$.
  Consequently, for every
  node~$\node \in \nodes$, \mbox{$\potential_{\node} \geq \lbPot_{\node}$} is
  satisfied.
  We note that the potentials satisfy the upper potential bounds
  due to the specific choice of the considered point~$(\loadFlowVec,
  \massflow, \potential)$.
  Hence, Constraints~\eqref{eq:expansion:potential-bounds} are also
  satisfied and the point~$(\loadFlowVec,
  \massflow, \potential)$ is feasible for
  Constraints~\eqref{eq:expansion:massflow-conservation}--%
  \eqref{eq:expansion:potential-bounds}.
  Since $\loadFlowVec$ is an arbitrary \loadscenario
  in~$\uncertaintySet$, this concludes the proof.
\end{proof}

We now extend the obtained characterization of robust feasibility to
the case that the expansion decision~$x \in X$ leads to an expanded
graph~$\graph'(\expVar)$ that has at least two connected components.
To this end, for
a given connected component~$\connectedComponent^{i}=(\nodes^{i},
\arcs^{i})\footnote{For the ease
    of presentation, we write $\connectedComponent^{i}=(\nodes^{i},
    \arcs^{i})$ instead of
    $\connectedComponent^{i}(\expVar)=(\nodes^{i}(\expVar),
    \arcs^{i}(\expVar))$ in the following.}$, we consider another
  auxiliary problem for computing the
maximal absolute flow that has to be transported between
the connected component~$\connectedComponent^{i}$ and the remaining
network in the uncertainty set.
This problem reads
\begin{equation}
  \label{eq:maximal-absolute-loadflow-connectedComponents}
  \maxAbsFlow_{\connectedComponent^{i}}(\expVar) \define \max_{\loadFlowVec}
  \abs{y}  \quad \st \quad y = \sum_{\node \in \nodes^{i} \cap
  \entries} \loadFlowVec_{\node} \rev{+}  \sum_{\node \in \nodes^{i} \cap
  \exits} \loadFlowVec_{\node}, \ \loadFlowVec \in \uncertaintySet.
\end{equation}
The value~$\maxAbsFlow_{\connectedComponent^{i}}(\expVar)$ is positive
if and only if there is a \loadscenario~$\loadFlowVec \in
\uncertaintySet$
with excess demand or excess supply regarding the connected
component~$\graph^{i}$.
In this case, $\maxAbsFlow_{\connectedComponent^{i}}(\expVar) > 0$,
the expansion decision~$x \in X$ is robust infeasible.

\begin{lemma}
  \label{lemma:necessary-condition-connected-components}
  Let $\expVar \in X$ be fixed and $\graph'(\expVar)=(\nodes,
  \rev{\arcs(x)})$ be the expanded graph.
  Furthermore, let~$\mathcal{\graph'}(\expVar)
  \define \set{\connectedComponent^{1}, \ldots,
    \connectedComponent^{n}}$ with
  $\connectedComponent^{i}=(\nodes^{i}, \arcs^{i})$ be the set
  of
  connected components of the expanded graph~$\graph'(\expVar)$.
  Then, Constraints~\eqref{eq:expansion:quantifiers-robust}--%
  \eqref{eq:expansion:potential-bounds} can only be satisfied
  \wrt~$\expVar$ if for every connected
  component~$\connectedComponent^{i}$ with $i \in \set{1,\ldots,n}$ of
  the expanded
  network~$\graph'(\expVar)$, the maximum excess demand or excess supply is
  zero, \ie, $\maxAbsFlow_{\connectedComponent^{i}}(\expVar) = 0$.
\end{lemma}
\begin{proof}
  If there is a connected component~$\connectedComponent^{i}$ with $i
  \in \set{1,\ldots,n}$ so
  that~$\maxAbsFlow_{\connectedComponent^{i}}(\expVar) > 0$ holds,
  then
  there is a scenario~$\loadFlowVec \in \uncertaintySet$ such that
  there is excess demand or excess supply in~$\connectedComponent^{i}$.
  Consequently, this scenario cannot be transported through the network
  since mass flow conservation~\eqref{eq:expansion:massflow-conservation}
  cannot be satisfied in the connected component~$\connectedComponent^{i}$.
\end{proof}
We note that if the graph consists only of a single connected
component, then it directly follows that the optimal objective value
of
Problem~(\ref{eq:maximal-absolute-loadflow-connectedComponents})
is zero because we only consider balanced \rev{load scenarios} in the
uncertainty set~(\ref{eq:demand-uncertainty-set}).
Using the previous lemma, we now extend
Theorem~\ref{theorem:characterization-single-connected-component} to
the case of multiple connected components in the expanded graph.

\begin{theorem}
  \label{theorem:characterization-multiple-connected-components}
  Let $\expVar \in X$ be fixed and $\graph'(\expVar)=(\nodes, \rev{\arcs(x)})$ be the
  expanded graph. Furthermore, let~$\mathcal{\graph'}(\expVar)
  \define \set{\connectedComponent^{1}, \ldots,
    \connectedComponent^{n}}$ with
  $\connectedComponent^{i}=(\nodes^{i}, \arcs^{i})$ be the set of
  connected components of the expanded graph~$\graph'(\expVar)$.
  Then, Constraints~\mbox{\eqref{eq:expansion:quantifiers-robust}--%
    \eqref{eq:expansion:potential-bounds}}
  are satisfied \wrt~$\expVar$ if and only if
  \begin{subequations} \label{characterization-passive-case-multiple-components}
  \begin{align}
        \maxAbsFlow_{\connectedComponent^{i}}(\expVar) = 0 \quad
    & \forAll \connectedComponent^{i} \in \mathcal{\graph'}(\expVar),
    \label{eq:maximal-absolute-loadflow-multiple-connected-components}
    \\
    \maxPotDiff_{\node, \otherNode}(\expVar) \leq \ubPot_{\node} -
    \lbPot_{\otherNode}
    \quad & \forAll (\node, \otherNode) \in (\nodes^{i})^{2} \text{ \rev{ with $\node \neq
    \otherNode$}}, \
            \connectedComponent^{i} \in \mathcal{\graph'}(\expVar),
            \label{ineq:maxPot-satisfies-bounds-multiple-connected-components}
            \\
    \minArcFlow_\arc(\expVar) \geq \lbMassflow_\arc
    \quad
    & \forAll \arc \in \arcs^{i}, \ \connectedComponent^{i} \in \mathcal{\graph'}(\expVar),
      \label{ineq:minimum-flow-satisfies-bounds-multiple-connected-components}
    \\
    \maxArcFlow_{\arc}(\expVar) \leq \ubMassflow_{\arc}
    \quad
    & \forAll \arc \in \arcs^{i}, \ \connectedComponent^{i} \in \mathcal{\graph'}(\expVar),
            \label{ineq:maximum-flow-satisfies-bounds-multiple-connected-components}
  \end{align}
  \end{subequations}
  holds.
\end{theorem}
\begin{proof}
  For a given expansion~$x \in X$, let the
  Constraints~\eqref{eq:expansion:quantifiers-robust}--%
  \eqref{eq:expansion:potential-bounds} be satisfied.
  Then, from
  Lemma~\ref{lemma:necessary-condition-connected-components}, it
  follows that
  Condition~\eqref{eq:maximal-absolute-loadflow-multiple-connected-components}
  holds.
  Hence, every $\loadFlowVec \in \uncertaintySet$ is balanced
  \wrt\ each connected component, \ie,
  $\sum_{\node \in \entries \cap \nodes^{i}} \loadFlowVec_{\node} \rev{+}
  \sum_{\node \in \exits \cap \nodes^{i}} \loadFlowVec_{\node} \rev{=0}$ for
  each $i \in \set{1,\ldots,n}$.
  Consequently, we can apply
  Theorem~\ref{theorem:characterization-single-connected-component} to
  each connected component~$\connectedComponent^{i}$ while using
  as uncertainty set the original uncertainty set projected onto
  the nodes~$\nodes^{i}$ of the connected component.
  This proves that
  Conditions~\eqref{characterization-passive-case-multiple-components}
  are satisfied.

  For fixed expansion~$\expVar \in X$, we now assume that
  Conditions~\eqref{characterization-passive-case-multiple-components}
  hold.
  Since
  Conditions~\eqref{eq:maximal-absolute-loadflow-multiple-connected-components}
  are satisfied, every $\loadFlowVec \in \uncertaintySet$ is balanced
  \wrt\ each connected component.
  Consequently, for each connected component, we can apply
  Lemma~\ref{lemma:uniquness-of-flows-passive-case}.
  Thus, for each \loadscenario~$\loadFlowVec \in U$, there are
  potentials~$\potential'$ and unique flows~$\massflow$ such
  that the set of feasible points satisfying
  Constraints~\eqref{eq:expansion:massflow-conservation}--\eqref{ineq:expansion:potential-coupling-candidateArcs-lBound}
  \wrt~$\expVar$ and
  $\massflow_{\arc} = 0$ for each arc \rev{not built, \ie, $\arc \in \bar{\arcs}(x)$,}
  is non-empty and given by
  \begin{equation*}
    \Defset{(q, \potential)}{(\potential_{\node})_{\node \in
        \nodes^{i}} = (\potential'_{\node} + \eta_{i})_{\node \in
        \nodes^{i}}, \ \eta_{i} \in \reals, \ i \in \abs{\mathcal{\graph'}(\expVar)}}.
  \end{equation*}
  Using this statement, we can apply the second part of the proof of
  Theorem~\ref{theorem:characterization-single-connected-component} to
  every connected component~$G^{i}$, which proves the claim.
\end{proof}

For a given expansion decision,
Theorem~\ref{theorem:characterization-multiple-connected-components}
allows to verify robust feasibility by solving at most
$\abs{\nodes} + \abs{\nodes}^{2} + 2\abs{\arcs}$ many nonlinear
optimization problems.
Furthermore, in case of robust infeasibility of the expansion
decision, solving these problems provides violating scenarios in~$U$
that render the expansion decision infeasible.
Consequently,
Theorem~\ref{theorem:characterization-multiple-connected-components}
resolves the main challenges~\ref{adversarial:robust-feasibility}
and~\ref{adversarial:computing-infeasible-scn} when applying the  adversarial approach to
the considered adjustable robust mixed-integer nonlinear optimization
problem~(\ref{expansion}).
We note that for checking robust feasibility, we have to solve the
nonconvex adversarial problems to global optimality.
In doing so, also the specific choice of the potential-based flow
model as well as the choice of the uncertainty set influence the
computational complexity of this task.
However, in the conducted computational study the adversarial problems
are solved rather fast and the
MINLPs~\eqref{expansion} pose a much
bigger computational challenge.

\begin{remark}
  \label{remark:bilevel-passive-network}
  \rev{Let us use the adversarial
    problem~\eqref{eq:maximum-potential-difference} to illustrate that
    the adversarial problems are equivalent to specifically chosen
    bilevel optimization problems.
    More precisely, the set of solutions to the single-level
    problem~\eqref{eq:maximum-potential-difference} is the same as the
    one for the bilevel problem
    \begin{align}
      \label{eq:adversarial-bilevel-problem}
      \max_{\loadFlowVec \in \uncertaintySet} \, \min_{\massflow, \potential}
      \quad \potential_{\node} -
      \potential_{\otherNode} \quad \st \quad
      \eqref{eq:expansion:massflow-conservation}\text{--}%
      \eqref{ineq:expansion:potential-coupling-candidateArcs-lBound}.
    \end{align}
    The intuition behind this bilevel problem is that the upper-level
    player chooses the worst-case load scenario~$\loadFlowVec$ in the uncertainty
    set~$\uncertaintySet$ with the goal to maximize the potential
    difference between the nodes~$u$ and~$v$.
    In contrast to this, the lower-level player, \eg, the transmission
    system operator, tries to minimize this potential difference.
    If the upper-level player finds a load scenario~$d$ for which the
    lower-level player cannot operate the network such that the potential
    difference stays
    within the potential bounds, \ie, $\potential_{\node}
    -\potential_{\otherNode} \leq \ubPot_\node -\lbPot_\otherNode$, then
    the current network design is robust infeasible.
    The equivalence of the single-level adversarial
    problem~\eqref{eq:maximum-potential-difference} and the bilevel
    problem~\eqref{eq:adversarial-bilevel-problem} follows from
    Lemma~\ref{lemma:uniquness-of-flows-passive-case} because, for
    a given load scenario~$\loadFlowVec$, the flows and potential differences
    are uniquely determined.
    Consequently, for a given upper-level decision, the lower-level
    variables are determined
    by~\eqref{eq:expansion:massflow-conservation}--%
    \eqref{ineq:expansion:potential-coupling-candidateArcs-lBound}
    up to a constant shift,
    \ie, they are predetermined by physics.
    This enables us to merge the upper- and lower-level problem
    of~\eqref{eq:adversarial-bilevel-problem}, which then leads to the
    equivalent single-level
    reformulation~\eqref{eq:maximum-potential-difference}.}

  \rev{We emphasize that Problem~\eqref{eq:adversarial-bilevel-problem} has a
    nonconvex lower-level problem and, thus, no duality-based single-level
    reformulations can be directly applied.
    Consequently, the derived adversarial
    problem~\eqref{eq:maximum-potential-difference} is a single-level
    reformulation of the challenging nonconvex bilevel
    problem~\eqref{eq:adversarial-bilevel-problem} that can only be
    derived by exploiting structural properties of potential-based flows.}
\end{remark}

\begin{remark} \label{remark:active-elements}
  \rev{Potential networks that are used to transport fluids over long
    distances often contain controllable elements, \eg, compressors and
    control valves in gas networks or pumps in water networks.
    These elements enable the system operator to actively
    decrease or increase potential levels.
    For representing such controllable elements, various models, ranging
    from linear ones, see, \eg,~\textcite{Sanches_et_al:2016, Plein2021}, to
    sophisticated nonlinear ones, see, \eg,
    \textcite{DAmbrosio_et_al:2015}, exist.
    However, even in case of simplified linear models for these
    elements, the characterization of
    Theorem~\ref{theorem:characterization-multiple-connected-components}
    does not hold.
    This directly follows from the discussion and the example of Section~3
    in~\textcite{Plein2021}, in which a booking can be interpreted as
    a specific box uncertainty set~$U$.
  It also becomes clear when considering the bilevel
  formulation~\eqref{eq:adversarial-bilevel-problem} of the
  adversarial problem.
  Including controllable elements leads to new variables and
  constraints in the lower level of the bilevel
  problem~\eqref{eq:adversarial-bilevel-problem}.
  However, for a given load scenario~$\loadFlowVec$, there can now
  be different controls of the controllable elements that lead to
  different flows.
  Consequently, the uniqueness result of
  Lemma~\ref{lemma:uniquness-of-flows-passive-case} does not hold
  anymore.
  Thus, in case of controllable elements, computing an
  adversarial problem leads to solving a challenging bilevel
  problem with a nonconvex lower-level problem,
  which cannot be directly reformulated as a single-level
  optimization problem as in~\eqref{eq:maximum-potential-difference}.
}

\end{remark}

\begin{algorithm}[!ht]
  \caption{Adversarial approach to solve the network design problem~\eqref{expansion}}
  \label{alg:adversarial-approach}
  \DontPrintSemicolon

  \KwInput{A Graph~$\graph=(\nodes, \exArcs \cup \expArcs)$ and an uncertainty
    set~$\uncertaintySet$ satisfying~(\ref{eq:demand-uncertainty-set}).}

  \KwOutput{An optimal adjustable robust expansion~$\expVar \in X$ for
    Problem~\eqref{expansion} or an indication of infeasibility.}
  Determine a finite set of scenarios~$S \subseteq \uncertaintySet$.

  Solve Problem~\eqref{expansion} \wrt~$S$ (instead
  of~$\uncertaintySet$) to get $(\expVar,
  \massflow, \potential)$\label{alg:line:solve-master-problem}.
  \label{alg:line-solve-expansion-problem}

  \If{the problem is infeasible}
  {\Return The problem is infeasible.}
  \label{alg:line-begin-check-robust-feasibility}

  Determine the set of all connected
  components~$\mathcal{\graph'}(\expVar)$ of the expanded graph
  $\graph'(\expVar)=(\nodes, \rev{\arcs(x)})$.

  \For{$\connectedComponent^{i} \in \mathcal{\graph'}(\expVar)$}{
    Solve
    Problem~\eqref{eq:maximal-absolute-loadflow-connectedComponents}
    to get $\loadFlowVec'$ with objective
    value~$\maxAbsFlow_{\connectedComponent^{i}}(\expVar)$.

    \If{$\maxAbsFlow_{\connectedComponent^{i}}(\expVar) > 0$ }{
      $S = S \cup \{\loadFlowVec'\}$ and go to
      Line~\ref{alg:line:solve-master-problem}.}}

  Set $\varphi^{\max} = 0$.

  \For{$\connectedComponent^{i} \in \mathcal{\graph'}(\expVar)$}{
    \For{$(\node, \otherNode) \in (\nodes^{i})^{2}$ \rev{ with $\node \neq
    \otherNode$}}{
      Solve Problem~\eqref{eq:maximum-potential-difference}
      \wrt\ $\connectedComponent^{i}$ to get
      $(\loadFlowVec', \massflow', \potential')$ with objective
      value~$\maxPotDiff_{\node,\otherNode}(\expVar)$.

      \If{$\maxPotDiff_{\node,\otherNode}(\expVar) > \ubPot_\node -
        \lbPot_\otherNode$ and
        $\maxPotDiff_{\node,\otherNode}(\expVar) - (\ubPot_\node -
        \lbPot_\otherNode) > \varphi^{\max}$}{
        Set $\varphi^{\max} = \maxPotDiff_{\node,\otherNode}(\expVar)
        - (\ubPot_\node - \lbPot_\otherNode)$
        and $\loadFlowVec^{\max} = \loadFlowVec'$.}
    }
  }

  \If{$\varphi^{\max} > 0$}{
        $S = S \cup \{\loadFlowVec^{\max}\}$ and go to Line~\ref{alg:line:solve-master-problem}.}

  Set $\massflow^{\max} = 0$.

  \For{$\connectedComponent^{i} \in \mathcal{\graph'}(\expVar)$}{
    \For{$\arc \in \arcs^{i}$}{
      Solve Problem~\eqref{eq:minimum-arc-flow}
      \wrt\ $\connectedComponent^{i}$ to get
      $(\loadFlowVec',\massflow', \potential')$ with objective
      value~$\minArcFlow_{\arc}(\expVar)$.

      \If{$\minArcFlow_{\arc}(\expVar) < \lbMassflow_{\arc}$
      and $\lbMassflow_{\arc}- {\minArcFlow_{\arc}(\expVar) } > \massflow^{\max}$}{
         Set $\massflow^{\max} = \lbMassflow_{\arc}- {\minArcFlow_{\arc}(\expVar) }$
         and $\loadFlowVec^{\max} = \loadFlowVec'$.}

      Solve Problem~\eqref{eq:maximum-arc-flow}
      \wrt\ $\connectedComponent^{i}$ to get
      $(\loadFlowVec', \massflow', \potential')$ with objective
      value~$\maxArcFlow_{\arc}(\expVar)$.

      \If{$\maxArcFlow_{\arc}(\expVar) > \ubMassflow_{\arc}$
      and ${\maxArcFlow_{\arc}(\expVar) } - \ubMassflow_{\arc} > \massflow^{\max}$}{
        Set $\massflow^{\max} = {\maxArcFlow_{\arc}(\expVar) } - \ubMassflow_{\arc}$
        and $\loadFlowVec^{\max} = \loadFlowVec'$.}

    }
  }
  \If{$\massflow^{\max} > 0$}{
   $S = S \cup \{\loadFlowVec^{\max}\}$ and go to Line~\ref{alg:line:solve-master-problem}.}
  \Return Optimal adjustable robust network design~$\expVar \in X$.
\end{algorithm}%

Embedding the results of
Theorem~\ref{theorem:characterization-multiple-connected-components}
into the adversarial approach leads to
Algorithm~\ref{alg:adversarial-approach}.
We note that there are multiple possibilities on how to integrate the
characterization of robust feasibility of
Theorem~\ref{theorem:characterization-multiple-connected-components}
in an adversarial approach.
In our implementation of Algorithm~\ref{alg:adversarial-approach},
we aim to keep the size of the MINLP~\eqref{expansion} \wrt~$S$ as small as
possible since solving this MINLP is computationally challenging.
Since the size of this problem increases with the size of the
scenario set~$S$, for an infeasible expansion decision, we only add
a single violating scenario to cut off this robust infeasible point.
More precisely, we first solve the adversarial
problems~\eqref{eq:maximal-absolute-loadflow-connectedComponents}
since these problems are typically less challenging than
Problems~\eqref{eq:maximum-potential-difference}--\eqref{eq:maximum-arc-flow},
which
contain the constraints of the nonconvex potential-based flows.
If solving
Problems~\eqref{eq:maximal-absolute-loadflow-connectedComponents}
leads to violating scenarios, \ie,
$\maxAbsFlow_{\connectedComponent^{i}}(\expVar) > 0$ holds, then we
add this scenario to cut off the robust infeasible
expansion decision~$x$ and start a new iteration.
Otherwise, we solve the adversarial
problems~\eqref{eq:maximum-potential-difference},
respectively~\eqref{eq:minimum-arc-flow}
and~\eqref{eq:maximum-arc-flow},  and add
a most violating scenario to the set of \rev{load scenarios}~$S$ if
applicable.
In general, it is also possible to stop solving these adversarial
problems after a first violating scenario is computed as for the case
of Problems~\eqref{eq:maximal-absolute-loadflow-connectedComponents}.
However, preliminary computational results showed that adding a most
violating scenario \wrt~\eqref{eq:maximum-potential-difference} leads
to a lower number of iterations of the algorithm.
We finally note that all adversarial
problems~\eqref{eq:maximal-absolute-loadflow-connectedComponents} and
\eqref{eq:maximum-potential-difference}--\eqref{eq:maximum-arc-flow}
can also be solved in parallel since they do not depend on each other.

We conclude this section with a positive answer for the main
challenge~\ref{adversarial:finite-termination}.

\begin{theorem}
  \label{theorem:adversarial-approach-finite-termination}
  Algorithm~\ref{alg:adversarial-approach} terminates after a finite
  number of iterations and either returns an adjustable robust
  solution of Problem~\eqref{expansion} or proves its infeasibility.
\end{theorem}
\begin{proof}
  If we consider a robust infeasible
  expansion decision~$\expVar \in X$, \ie, it cannot be extended
  to satisfy Constraints~\eqref{eq:expansion:quantifiers-robust}--%
  \eqref{eq:expansion:potential-bounds} for
  all~$\loadFlowVec \in \uncertaintySet$,
  then there exists a \loadscenario~$\loadFlowVec \in \uncertaintySet$ that
  violates one of the conditions
  in~\eqref{characterization-passive-case-multiple-components}.
  Due to the construction of the algorithm, one of these violating
  \rev{load scenarios} is added to the set of scenarios~$S$ if~$x$ is
  part of the
  optimal solution in Line~\ref{alg:line:solve-master-problem}.
  Consequently, the considered network expansion~$\expVar \in X$ is excluded
  in the next iteration.
  Thus, the algorithm terminates after a finite number of iterations
  because we only have a finite number of possible assignments
  for~$\expVar \in X \subseteq \set{0,1}^{\abs{\expArcs}}$.
  Since Problem~\eqref{expansion} \wrt~$S$ is a relaxation of
  Problem~\eqref{expansion} \wrt~$\uncertaintySet$, the algorithm
  either correctly returns an optimal solution or correctly verifies
  infeasibility.
\end{proof}

\section{Enhanced Solution Techniques}
\label{sec:enhanced-solution-techniques}

When applying Algorithm~\ref{alg:adversarial-approach}, there are two
main challenges from the computational point of view.
For verifying robust feasibility, the adversarial
problems~(\ref{eq:maximum-potential-difference})--(\ref{eq:maximum-arc-flow})
have to be solved to global optimality.
In particular, solving~$\abs{\nodes}^{2}$ many
problems~(\ref{eq:maximum-potential-difference}) can be
computationally expensive.
In addition, solving the MINLP~(\ref{expansion}) \wrt\ the worst-case
scenarios~$S$ becomes more demanding from iteration to
iteration due to the increasing set of scenarios~$S$.
In the following, we present different techniques that address these
computational challenges.

\subsection{Reducing the Number of Adversarial Problems}

We now prove that under specific assumptions on the potential bounds,
we can significantly reduce the number of adversarial
problems~(\ref{eq:maximum-potential-difference})
that have to be solved to verify robust feasibility.
The intuition is based on the observation that
in the considered potential-based flow setting, there is always a
source node with maximal potential level and a sink node with minimal
potential level.
\begin{observation}
  \label{obs:max-pressure-entry-low-pressure-exit}
  Let $\expVar \in X$ be fixed and
  let~$\connectedComponent^{i}=(\nodes^{i}, \arcs^{i})$ be a connected
  component of the expanded graph~$\graph'(\expVar)=(\nodes,
  \rev{\arcs(x)})$.
  Further, let the point~$(\loadFlowVec, \massflow, \potential)$ satisfy
  \mbox{Constraints~\eqref{eq:expansion:massflow-conservation}--%
  \eqref{ineq:expansion:potential-coupling-candidateArcs-lBound}}
  \wrt~$\connectedComponent^{i}$.
  Then, there is a source node~$w \in \entries^{i} \define \entries
  \cap \nodes^{i}$ with $\potential_{w} = \max_{v \in \nodes^{i}}
  \potential_{v}$ and a sink node~$u \in \exits^{i} \define \exits
  \cap \nodes^{i}$ with $\potential_{u} = \min_{v \in \nodes^{i}}
  \potential_{v}$.
\end{observation}
This observation follows from the assumption that for every
arc~$\arc \in \arcs$, the potential
function~$\potFunc_{\arc}$ is strictly increasing.
Consequently, sending flow from a source to a sink node leads to a
positive potential drop.
Using this observation, we now prove that under specific requirements
for the potential bounds, we only have to compute the maximum
potential difference, \ie, solve
Problem~(\ref{eq:maximum-potential-difference}), between sources and
sinks.
\begin{lemma}
  \label{lemma:entry-exit-maximum-pot-diff}
  Let $\expVar \in X$ be fixed and
  let~$\connectedComponent^{i}=(\nodes^{i}, \arcs^{i})$ be a connected
  component of the expanded graph~$\graph'(\expVar)=(\nodes,
  \rev{\arcs(x)})$.
  For each source~$w \in \entries^{i} \define \entries
  \cap \nodes^{i}$,
  let the upper potential bound satisfy
  $\ubPot_{w} \leq \ubPot_{v}$ for all sinks and inner
  nodes~$v \in (\exits \cup \innodes) \cap \nodes^{i}$.
  For each sink~$\node \in \exits^{i} \define \exits
  \cap \nodes^{i}$, let the
  lower potential bound satisfy~$\lbPot_{\node} \geq \lbPot_{v}$ for
  all sources and inner nodes~$v \in (\entries \cup \innodes)
  \cap \nodes^{i}$.
  Then,
  \begin{equation}
    \label{eq:entry-exit-max-pot}
    \maxPotDiff_{\node, \otherNode}(\expVar) \leq \ubPot_{\node} -
    \lbPot_{\otherNode}
    \quad  \forAll (\node, \otherNode) \in \entries^{i} \times \exits^{i}
  \end{equation}
  implies
  \begin{equation*}
    \maxPotDiff_{\node, \otherNode}(\expVar) \leq \ubPot_{\node} -
    \lbPot_{\otherNode}
    \quad  \forAll (\node, \otherNode) \in (\nodes^{i})^{2}.
  \end{equation*}
\end{lemma}
\begin{proof}
  Let the inequalities in~\eqref{eq:entry-exit-max-pot} be
  satisfied.
  We now contrarily assume that there is a node pair~$(m,n) \in
  (V^{i})^{2} \setminus \entries^{i} \times \exits^{i}$ that satisfies
  $\maxPotDiff_{m, n}(\expVar) > \ubPot_{m}
  -\lbPot_{n}$.
  Hence, there exists a solution~$(\loadFlowVec, \massflow,
  \potential)$ of Problem~(\ref{eq:maximum-potential-difference})
  with $\potential_{m} - \potential_{n} > \ubPot_{m}
  -\lbPot_{n}$.
  From Observation~\ref{obs:max-pressure-entry-low-pressure-exit}, it
  follows that there is a source~$w \in \entries^{i}$ with $\potential_{w}
  = \max_{v \in \nodes^{i}} \potential_{v}$ and a sink~$u \in \exits^{i}$ with
  $\potential_{u} = \min_{v \in \nodes^{i}} \potential_{v}$.
  We now conduct a case distinction.

  If $m \in (\exits \cup \innodes) \cap \nodes^{i}$ and~$n \in
  (\entries \cup \innodes) \cap \nodes^{i}$, we obtain the
  contradiction
  \begin{equation*}
    \maxPotDiff_{m, n}(\expVar) = \potential_{m} - \potential_{n}
    \leq
    \potential_{w} - \potential_{u}
    \leq
    \maxPotDiff_{w, u}(\expVar)
    \leq
    \ubPot_{w} - \lbPot_{u}
    \leq
    \ubPot_{m} - \lbPot_{n},
  \end{equation*}
  where the last inequality follows from the assumptions on the
  potential bounds.
  Additionally, if $m \in (\exits \cup \innodes) \cap \nodes^{i}$ and
  $n \in \exits^{i}$, we obtain the contradiction
  \begin{equation*}
    \maxPotDiff_{m, n}(\expVar) = \potential_{m} - \potential_{n}
    \leq
    \potential_{w} - \potential_{n}
    \leq
    \maxPotDiff_{w, n}(\expVar)
    \leq
    \ubPot_{w} - \lbPot_{n}
    \leq
    \ubPot_{m} - \lbPot_{n}.
  \end{equation*}
  Finally, if $m \in \entries^{i}$ and $n \in (\entries \cup \innodes)
  \cap \nodes^{i}$, then we obtain the contradiction
  \begin{equation*}
    \maxPotDiff_{m, n}(\expVar) = \potential_{m} - \potential_{n}
    \leq
    \potential_{m} - \potential_{u}
    \leq
    \maxPotDiff_{m, u}(\expVar)
    \leq
    \ubPot_{m} - \lbPot_{u}
    \leq
    \ubPot_{m} - \lbPot_{n}. \qedhere
  \end{equation*}
\end{proof}
As a consequence of this lemma, we can reduce the maximal number of
adversarial problems~(\ref{eq:maximum-potential-difference}) that have
to be solved to check robust feasibility from
$\abs{\nodes}^{2}$ to at most $\abs{\entries} \times \abs{\exits}$ many
problems.
In the case of real-world utility networks, this reduction is
significant because usually there are only a small number of sources
in these networks.
Furthermore, the assumptions regarding the potential bounds
\rev{of Lemma~3} are often satisfied in utility networks such as gas
or water networks.

We additionally remark that we can add to the adversarial
problems~\eqref{eq:maximum-potential-difference} \wrt~$(u,v)$ the
constraint
\begin{equation*}
  \potential_{u} - \potential_{v} \geq \ubPot_u - \lbPot_v.
\end{equation*}
If this constraint renders the adversarial problem infeasible, then
we can directly conclude that there is no violating scenario \wrt~$(u,v)$.
Preliminary computational results have shown that this
approach significantly speeds up the
computational process.
Analogously, we can add the constraints~$\massflow_{\arc} \leq
\lbMassflow_\arc$ to Problem~\eqref{eq:minimum-arc-flow}
and~\mbox{$\massflow_{\arc} \geq \ubMassflow_\arc$} to
Problem~\eqref{eq:maximum-arc-flow}.

\subsection{Computing Lower Bounds}
\label{subsec:computing-lower-bounds}

We now focus on the algorithmic idea to iteratively update
a lower bound for the objective function of the
MINLP~\eqref{expansion} \wrt~$S$
by exploiting the structure of
Algorithm~\ref{alg:adversarial-approach}.
Thus, we add to the MINLP~\eqref{expansion} the constraint
\begin{equation}
  \label{eq:objective-bound-MINLP}
  \sum_{\arc \in \expArcs} c_{\arc}
  \expVar_{\arc}
  \geq
  \kappa,
\end{equation}
where $\kappa \in \reals_{\geq 0}$ is a valid lower bound of the
objective value of Problem~\eqref{expansion} that we iteratively
update.
Here, ``valid'' means that we do not cut off any optimal solution by
adding Constraint~\eqref{eq:objective-bound-MINLP}.

Since we increase the set of scenarios~$S$ in each iteration of
Algorithm~\ref{alg:adversarial-approach}, we can use the optimal
objective value of Problem~\eqref{expansion} of the previous
iteration, \ie, without the last added ``worst-case''
scenario~$\loadFlowVec'$, as a lower bound for the optimal objective
value in the next iteration.
Thus, we can iteratively set~$\kappa =
\minCostsRobustExpansion(S \setminus \set{d'})$, where
$\minCostsRobustExpansion(S
\setminus \set{d'})$ is the
optimal objective value of Problem~\eqref{expansion} \wrt\ the
scenario set~$S \setminus \set{d'}$.
We note that obtaining this lower bound is straightforward and
computationally cheap since we already solved the corresponding MINLPs
in Algorithm~\ref{alg:adversarial-approach}.
However, this bound can be improved since it dismisses
all information regarding the last added worst-case
scenario~$\loadFlowVec'$.
To do so, we now present two relaxations of the
MINLP~\eqref{expansion}
that can be solved prior to solving Problem~\eqref{expansion} to
improve the lower objective bound~$\kappa$.

First, we can solve the MINLP~\eqref{expansion} only \wrt\ the last
added ``worst-case'' scenario~$\loadFlowVec'$, which is a relaxation
of Problem~\eqref{expansion} due to $\set{\loadFlowVec'} \subseteq S$.
The benefit of this simple relaxation is that the size of the
corresponding MINLP, \ie, the number of variables and constraints, does
not increase from iteration to iteration, in contrast to
Problem~\eqref{expansion} \wrt\ the entire set~$S$.
In the following, we denote this relaxation as~\reducedHeur.
Our computational results of Section~\ref{sec:numerical-results}
indicate that this relaxation is particularly useful at the early
iterations in Algorithm~\ref{alg:adversarial-approach}.

Second, we apply a well-known mixed-integer second-order cone
relaxation for
gas networks, see~\textcite{Sanches_et_al:2016}, to the considered
general potential-based flows.
This relaxation leads to a mixed-integer convex problem if for each
arc~$\arc \in \arcs$,
the potential functions~$\potFunc_{\arc}$ are convex on the
domain~$\reals_{\geq 0}$.
This is the case, \eg, for water and gas networks;
see~\eqref{different-potential-functions}.
In line with~\textcite{Sanches_et_al:2016}, we start with an
equivalent reformulation of the
Constraints~\eqref{eq:expansion:potential-coupling-existingArcs}--%
\eqref{ineq:expansion:potential-coupling-candidateArcs-lBound}
using additional binary variables that indicate the flow direction.
To this end, for each arc~$\arc=(u,v,\labelArc) \in \arcs$ and
\loadscenario~$\loadFlowVec \in \uncertaintySet$, we
introduce a binary variable~$y_{\arc}^{\loadFlowVec} \in \set{0,1}$.
The binary variable equals one if the flow is from node~$u$ to
node~$v$ and otherwise, it is zero. This is ensured by the
constraints
\begin{equation}
  \label{eq:flow-dir-cons}
  \lbMassflow_{\arc} (1-y^{\loadFlowVec}_{\arc}) \leq
  \massflow^{\loadFlowVec}_{\arc} \leq
  \ubMassflow_{\arc} y^{\loadFlowVec}_{\arc}, \quad \arc \in \arcs,
\end{equation}
where $\massflow^{\loadFlowVec}$ are the corresponding flows of
Problem~\eqref{expansion}.
We note that for an arc flow of zero, \ie,
$\massflow_{\arc}^{\loadFlowVec} = 0$, the
variable~$y_{\arc}^{\loadFlowVec}$ can be chosen arbitrarily.
Moreover, we assume that for each arc~$\arc \in \arcs$, the flow
bounds satisfy $\lbMassflow_{\arc} \leq 0 \leq \ubMassflow_{\arc}$,
which is a natural assumption in the context of potential-based flows.
Using the introduced binary variables for the flow directions and the
symmetry of the potential functions, \ie,
$\potFunc_{\arc}(-\massflow_{\arc}^{\loadFlowVec}) = -
\potFunc_{\arc}^{\loadFlowVec}(\massflow_{\arc}^{\loadFlowVec})$, we
can
equivalently\footnote{Here, equivalent means in terms of the feasible
  expansion decisions.} reformulate
Constraints~\eqref{eq:expansion:potential-coupling-existingArcs}--%
\eqref{ineq:expansion:potential-coupling-candidateArcs-lBound}
as
\begin{subequations}
  \label{cons:potential-link-flow-dir-vars}
  \begin{align}
    & \quad
      (\potential^{\loadFlowVec}_{\otherNode} -
      \potential^{\loadFlowVec}_{\node})
      +
      2y_{\arc}^{\loadFlowVec}(\potential^{\loadFlowVec}_{\node} -
      \potential^{\loadFlowVec}_{\otherNode})
      =
      \potFunc_{\arc}(\abs{\massflow^{\loadFlowVec}_{\arc}}), \quad
      \arc=(\node,\otherNode, \labelArc)
      \in \exArcs,
      \label{cons:potential-link-flow-dir-vars-equality}
    \\
    & \quad
      (\potential^{\loadFlowVec}_{\otherNode} -
      \potential^{\loadFlowVec}_{\node})
      +
      2y_{\arc}^{\loadFlowVec}(\potential^{\loadFlowVec}_{\node} -
      \potential^{\loadFlowVec}_{\otherNode})
      \leq
      \potFunc_{\arc}(\abs{\massflow^{\loadFlowVec}_{\arc}}) +
      (1-\expVar_{\arc}) M^{+}_{\arc},
      \quad \arc=(\node,\otherNode, \labelArc)
      \in \expArcs,
      \label{cons:potential-link-flow-dir-vars-geq-ineq}
    \\
    & \quad
      (\potential^{\loadFlowVec}_{\otherNode} -
      \potential^{\loadFlowVec}_{\node})
      +
      2y_{\arc}^{\loadFlowVec}(\potential^{\loadFlowVec}_{\node} -
      \potential^{\loadFlowVec}_{\otherNode})
      \geq
      \potFunc_{\arc}(\abs{\massflow^{\loadFlowVec}_{\arc}}) +
      (1-\expVar_{\arc}) M^{-}_{\arc},
      \quad \arc=(\node,\otherNode, \labelArc)
      \in \expArcs.
      \label{cons:potential-link-flow-dir-vars-leq-ineq}
  \end{align}
\end{subequations}
Analogously to~\eqref{eq:big-Ms-expansion-model}, for each
arc~$\arc=(u,v,\labelArc) \in \expArcs$, we
adapt the big-$M$
values to~\mbox{$M_{\arc}^{+} = \max \set{\ubPot_{u}-\lbPot_{v},
    \ubPot_{v}-\lbPot_{u}}$} and
$M_{\arc}^{-} =  \min \set{\lbPot_{u}-\ubPot_{v},
  \lbPot_{v}-\ubPot_{u}}$.
We note that the bilinear terms on the left-hand side of
Constraints~(\ref{cons:potential-link-flow-dir-vars}) can be
linearized using the inequalities of~\textcite{McCormick1976}; see
Appendix~\ref{sec:appendix} for the corresponding reformulations.
Using these constraints we can equivalently represent the expansion
problem~\eqref{expansion} by the MINLP
\begin{equation}
  \label{eq:expansion-with-flow-directions}
  \minCostsRobustExpansion(\uncertaintySet) \define
  \min_{\expVar, \massflow, \potential} \quad
  \sum_{\arc \in \expArcs} c_{\arc}
  \expVar_{\arc} \quad \st \quad
  \text{\eqref{eq:expansion:expansion-constraints}--\eqref{eq:expansion:massflow-conservation}}, \,
  \text{\eqref{eq:expansion:massflow-bounds-exArcs}--\eqref{eq:expansion:potential-bounds}}, \,
  \eqref{eq:flow-dir-cons}, \, \eqref{cons:potential-link-flow-dir-vars}.
\end{equation}
Analogously to~\textcite{Sanches_et_al:2016}, we now relax this
problem by replacing~\eqref{cons:potential-link-flow-dir-vars-equality} by
inequalities and by dismissing~\eqref{cons:potential-link-flow-dir-vars-geq-ineq}.
This leads to the relaxation
\begin{subequations}
  \label{eq:expansion-with-flow-directions-relaxation}
  \begin{align}
    \min_{\expVar, \massflow, \potential}
    \quad & \sum_{\arc \in \expArcs} c_{\arc}
            \expVar_{\arc}
    \\
    \st \quad &
                \text{\eqref{eq:expansion:expansion-constraints}--\eqref{eq:expansion:massflow-conservation}}, \,
                \text{\eqref{eq:expansion:massflow-bounds-exArcs}--\eqref{eq:expansion:potential-bounds}}, \,
                \eqref{eq:flow-dir-cons}, \,
                \eqref{cons:potential-link-flow-dir-vars-leq-ineq},
    \\
          & (\potential^{\loadFlowVec}_{\otherNode} -
            \potential^{\loadFlowVec}_{\node})
            +
            2y_{\arc}^{\loadFlowVec}(\potential^{\loadFlowVec}_{\node} -
            \potential^{\loadFlowVec}_{\otherNode})
            \geq
            \potFunc_{\arc}(\abs{\massflow^{\loadFlowVec}_{\arc}}),
            \quad \arc=(\node,\otherNode, \labelArc)
            \in \exArcs.
  \end{align}
\end{subequations}
Using McCormick inequalities, this relaxation can again be
reformulated as a convex MINLP if the potential functions are convex
on the nonnegative domain.
In addition, for the later considered gas networks, this problem turns
into a mixed-integer second-order cone problem.
The corresponding reformulations are explicitly outlined in
Appendix~\ref{sec:appendix}.
Overall, in Algorithm~\ref{alg:adversarial-approach}, we can solve
Problem~\eqref{eq:expansion-with-flow-directions-relaxation} \wrt\ the
scenario set~$S$ at the beginning of each iteration to obtain a lower
bound for the objective value of the MINLP~\eqref{expansion}.
Our computational results show that we do not only obtain a
tight lower bound, but it is also often the case that
the relaxation provides a feasible and, thus, optimal point for the
MINLP~\eqref{expansion}. For more details see
Section~\ref{sec:numerical-results}.

We finally remark that we also add the obtained lower bounds for the
objective value of the MINLP~\eqref{expansion} to the described
relaxations as well.
The important difference is that adding these bounds possibly cut
off solutions of the relaxations.
However, it will preserve all optimal solutions of the
MINLP~\eqref{expansion} \wrt~$U$ in the feasible region of the
relaxations.
Thus, adding these lower bounds for the objective value can strengthen
the presented relaxations.

\subsection{Acyclic Inequalities}
\label{subsec:acyclic-inequalities}

We now brief\/ly review the valid inequalities for potential-based
flows derived in~\textcite{Habeck_and_Pfetsch:2022}.
Adding these inequalities to the MINLP~\eqref{expansion} significantly
speeds up the computations.
These valid constraints exploit that in the considered setting of
potential-based flows, there cannot be any cyclic flow.
To see this, let $C$ be a cycle in the undirected version of the
network~$G=(\nodes, \arcs)$.
Considering this cycle in the original directed graph~$G$ leads to two
subsets of arcs~$C_{1}, C_{2} \subseteq \arcs$.
Here, $C_{1}$ represents the corresponding forward arcs of the cycle
and $C_{2}$ represents the backward arcs, \ie, those arcs have the
opposite direction in the original graph.
Summing up the corresponding potential
constraints~\eqref{eq:expansion:potential-coupling-existingArcs} along
the cycle leads to
\begin{equation*}
  \sum_{\arc \in C_{1}} \potFunc_\arc
  (\massflow^{\loadFlowVec}_{\arc})
  -
  \sum_{\arc \in C_{2}} \potFunc_\arc
  (\massflow^{\loadFlowVec}_{\arc})
  =
  \sum_{\arc=(u,v,\labelArc) \in C_{1}} \potential_{u} -
  \potential_{v}
  -
  \sum_{\arc=(u,v,\labelArc) \in C_{2}} \potential_{u} -
  \potential_{v} = 0.
\end{equation*}
Since the potential functions are strictly increasing and symmetric
\wrt~zero, this implies that there cannot be any cyclic flow.
As described in~\textcite{Habeck_and_Pfetsch:2022}, using the flow
direction variables~$y_{\arc}^{\loadFlowVec}$, this
acyclic property can be translated to the valid inequalities

\begin{equation} \label{eq:acyclic-inequalities}
  \sum_{\arc \in C_{1}} y_{\arc}^{\loadFlowVec}
  + \sum_{\arc \in C_{2}} (1-y^{\loadFlowVec}_{\arc})
  \leq \abs{C}-1,
  \quad
  \sum_{\arc \in C_{1}} (1-y_{\arc}^{\loadFlowVec})
  + \sum_{\arc \in C_{2}} y^{\loadFlowVec}_{\arc}
  \leq \abs{C}-1.
\end{equation}
In our computational study, we add these valid inequalities not only
to the MINLP~\eqref{expansion}, but also to the relaxations of the
previous section since these relaxations preserve the acyclic property
of potential-based flows.

We emphasize that we can add these inequalities for each cycle
of the graph that contains all existing arcs and all candidate arcs.
This is based on the observation that if such a cycle contains an
arc~$\arc$
that is not built, then the corresponding flow
is zero and we can arbitrarily choose the flow direction
variable~$y_{\arc}^{\loadFlowVec}$.
Consequently, the corresponding acyclic
inequality~\eqref{eq:acyclic-inequalities} is redundant.

We also exploit these
acyclic inequalities to tighten the given
arc flow bounds \wrt~a given \loadscenario a priori to solving the
relaxations, respectively the MINLP~\eqref{expansion}.
To this end, for each \loadscenario~$\loadFlowVec \in S$ and arc~$\arc
\in \arcs$, we solve the
mixed-integer linear optimization problems
\begin{equation}
  \label{eq:mip-tighten-flow-bounds}
  \max_{\massflow^{\loadFlowVec}} \quad \massflow^{\loadFlowVec}_{\arc} \quad
  \st \quad \eqref{eq:expansion:massflow-conservation}, \,
  \eqref{eq:acyclic-inequalities},
  \quad
  \min_{\massflow^{\loadFlowVec}} \quad \massflow^{\loadFlowVec}_{\arc} \quad
  \st \quad \eqref{eq:expansion:massflow-conservation}, \,
  \eqref{eq:acyclic-inequalities},
\end{equation}
to obtain a upper and lower bound for the arc
flow~$\massflow_{\arc}^{\loadFlowVec}$.
We note that
Problem~\eqref{eq:mip-tighten-flow-bounds} is a simple uncapacitated
linear flow problem with the additional restriction that the flows are
acyclic.

\section{How Many Scenarios Do We Need?}
\label{sec:worst-case-scenarios}

We analyze the number of added worst-case scenarios in
Algorithm~\ref{alg:adversarial-approach} using an academic example.
The considered graph appears in similar ways as
subnetworks in many real-world utility networks.
On the one hand, we show that the considered topology can
theoretically lead to many different worst-case scenarios.
On the other hand, we highlight that under realistic assumptions on
the capacities of the sources, the latter most likely does not
occur, which we also empirically observe in our computational study.

We now consider the existing network~$\graph=(\nodes,
\arcs)$ with a single source~$\entries=\set{\rev{s}}$,
a single inner node~$\innodes=\set{0}$, and the
sinks~$\exits=\set{1,\ldots,n}$ with~$n \geq 2$, \ie, \mbox{$\nodes =
\entries \cup \exits \cup \innodes$}.
The arcs are given by~\mbox{$\arcs=\set{(\rev{s}, 0, \text{ex})} \cup
  \defset{(0,i, \text{ex})}{i \in \set{1,\ldots,n}}$}.
Here, ``ex'' represents the label for the existing arcs.
For the ease of presentation, we now focus on gas networks, \ie, we
consider the potential functions~$\potFunc_{\arc}(\massflow_{\arc}) =
\Lambda_{\arc} \massflow_{\arc} \abs{\massflow_{\arc}}$.
For each arc~$\arc \in \arcs$, we further
choose~$\Lambda_{\arc}=1$ and for each node~$w \in \nodes$, we set
the upper and lower potential bounds~$[\lbPot_{w},
\ubPot_{w}]=[1,5]$.
In addition, we dismiss arc flow bounds.
A visualization of this network is given by
Figure~\ref{fig:network-worst-case-scenario-example-A}.

\begin{figure}
  \captionsetup[subfigure]{justification=centering}
  \begin{subfigure}[t]{0.32\textwidth}
    \begin{center}
      \begin{tikzpicture}[thick]
\node[draw, circle, inner sep=0.07cm, fill=gray!30] (A) at (2,3.5) {$s$};
\node[draw, circle, inner sep=0.05cm, fill=gray!30] (B) at (2,2) {$0$};
\node[draw, circle, inner sep=0.05cm, fill=gray!30] (C) at (0.5,0.5) {$1$};
\node[draw, circle, inner sep=0.05cm, fill=gray!30] (D) at (1.5,0.5) {$2$};
\node[] (E) at (2.5,0.5) {$\ldots$};
\node[draw, circle, inner sep=0.06cm, fill=gray!30] (F) at (3.5,0.5) {$n$};
\draw[->, >=latex] (A) -- node[above=0.2cm] {} ++(B);
\draw[->, >=latex] (B) -- node[above=0.2cm] {} ++(C);
\draw[->, >=latex] (B) -- node[below=0.2cm] {} ++(D);
\draw[->, >=latex] (B) -- node[below=0.2cm] {} ++(F);
\end{tikzpicture}
      \caption{Existing network}
      \label{fig:network-worst-case-scenario-example-A}
    \end{center}
  \end{subfigure}
  \begin{subfigure}[t]{0.32\textwidth}
    \begin{center}
      \begin{tikzpicture}[thick]
\node[draw, circle, inner sep=0.07cm, fill=gray!30] (A) at (2,3.5) {$s$};
\node[draw, circle, inner sep=0.05cm, fill=gray!30] (B) at (2,2) {$0$};
\node[draw, circle, inner sep=0.05cm, fill=gray!30] (C) at (0.5,0.5) {$1$};
\node[draw, circle, inner sep=0.05cm, fill=gray!30] (D) at (1.5,0.5) {$2$};
\node[] (E) at (2.5,0.5) {$\ldots$};
\node[draw, circle, inner sep=0.06cm, fill=gray!30] (F) at (3.5,0.5) {$n$};
\draw[->, >=latex] (A) -- node[above=0.2cm] {} ++(B);
\draw[->, red, >=latex, dashed] (A.south) to [out=210,in=140] (B.north);
\draw[->, >=latex] (B) -- node[above=0.2cm] {} ++(C);
\draw[->, red, >=latex, dashed] (B.220) to [out=190,in=80] (C.53);
\draw[->, >=latex] (B) -- node[below=0.2cm] {} ++(D);
\draw[->, >=latex] (B) -- node[below=0.2cm] {} ++(F);
\end{tikzpicture}
      \caption{Expanded network after the first iteration}
      \label{fig:network-worst-case-scenario-example-B}
    \end{center}
  \end{subfigure}
  \begin{subfigure}[t]{0.32\textwidth}
    \begin{center}
      \begin{tikzpicture}[thick]
\node[draw, circle, inner sep=0.07cm, fill=gray!30] (A) at (2,3.5) {$s$};
\node[draw, circle, inner sep=0.05cm, fill=gray!30] (B) at (2,2) {$0$};
\node[draw, circle, inner sep=0.05cm, fill=gray!30] (C) at (0.5,0.5) {$1$};
\node[draw, circle, inner sep=0.05cm, fill=gray!30] (D) at (1.5,0.5) {$2$};
\node[] (E) at (2.5,0.5) {$\ldots$};
\node[draw, circle, inner sep=0.06cm, fill=gray!30] (F) at (3.5,0.5) {$n$};
\draw[->, >=latex, thick] (A) -- node[above=0.2cm] {} ++(B);
\draw[->, red, >=latex, dashed] (A.south) to [out=210,in=140] (B.north);
\draw[->, >=latex, thick] (B) -- node[above=0.2cm] {} ++(C);
\draw[->, red, >=latex, dashed] (B.220) to [out=190,in=80] (C.53);
\draw[->, >=latex, thick] (B) -- node[below=0.2cm] {} ++(D);
\draw[->, red, >=latex, dashed] (B.247) to [out=220,in=110] (D.78);
\draw[->, >=latex, thick] (B) -- node[below=0.2cm] {} ++(F);
\draw[->, red, >=latex, dashed] (B.311) to [out=280,in=170] (F.143);
\end{tikzpicture}
      \caption{Robust optimal solution}
      \label{fig:network-worst-case-scenario-example-C}
    \end{center}
  \end{subfigure}
  \caption{Academic network: existing arcs in solid black,
    expanded arcs in dashed red.}
  \label{fig:network-worst-case-scenario-example}
\end{figure}
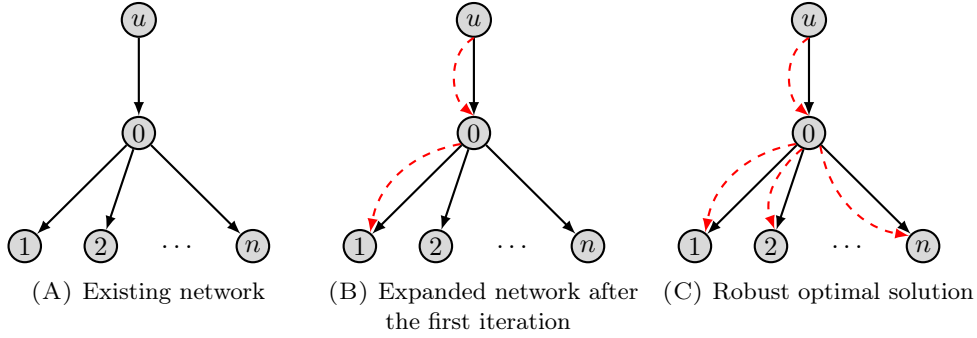

We now apply Algorithm~\ref{alg:adversarial-approach} to robustify the
existing network~$\graph$.
To do so, we have the expansion
candidates~$\expArcs=\set{(\rev{s}, 0, \text{ca})} \cup
\defset{(0,i, \text{ca})}{i \in \set{1,\ldots,n}}$ with
$\Lambda_{\arc} =1, \arc \in \expArcs$, \ie, for each
arc, we have an identical expansion arc in parallel to the existing
one.
Further, we consider the box uncertainty set
\begin{equation*}
  U = \Defset{\loadFlowVec \in \reals^{\nodes}}{\loadFlowVec_{\rev{s}}
    \rev{+} \sum_{\otherNode \in \exits}
    \loadFlowVec_{\otherNode} \rev{=0},
    \ \loadFlowVec_{w} \in [0,2], w \in \nodes
    \rev{\setminus{\set{s}}},
    \ \rev{\loadFlowVec_{s} \in [-2,0],}
    \ \loadFlowVec_{0} = 0}.
\end{equation*}
For applying Algorithm~\ref{alg:adversarial-approach} to this
instance, we have to solve in each iteration the adversarial
problems~(\ref{eq:maximum-potential-difference}).
Afterward, if applicable, we add the scenario that violates the
corresponding potential bounds most.

\rev{Overall, Algorithm~\ref{alg:adversarial-approach} terminates after
  $n$~iterations with a robust  network, in which every candidate arc is
  built; see Figure~\ref{fig:network-worst-case-scenario-example-C}.
  In every iteration, we consider a different
  worst-case scenario~$\loadFlowVec \in U$ given
  by~$-\loadFlowVec_{\rev{s}} = \loadFlowVec_{\otherNode}= 2$ with
  $\otherNode \in \set{1,\ldots,n}$ and the
  remaining loads are zero.
  This leads to expanding the unique path from $s$ to $\otherNode$;
  see, \eg, Figure~\ref{fig:network-worst-case-scenario-example-B} in
  which, w.l.o.g., $v=1$ is assumed.
  Consequently, to obtain a robust network, the algorithm considered the
  following set of load scenarios
  \begin{equation*}
    S = \set{-\loadFlowVec_{\rev{s}} = \loadFlowVec_{\otherNode}=2, \
      \loadFlowVec_{w} = 0, w \in \exits \setminus \set{\otherNode}
      \forAll
      \otherNode \in \exits},
  \end{equation*}
  which has a cardinality of~$\abs{\exits}$.
  Thus, the number of worst-case scenarios scales
  with the number of sinks in the considered network.}

However, this relatively large number of scenarios is based on the
very small capacity of the source that can only satisfy
the maximal demand of a single sink at once.
To see this, we now consider an adapted instance in which the
capacity of the source suffices to satisfy all demands of the exits at
once, \ie, the uncertainty set is given by
\begin{equation*}
  \tilde{U} = \Defset{\loadFlowVec \in
    \reals^{\nodes}}{\loadFlowVec_{\rev{s}} \rev{+} \sum_{\otherNode \in
      \exits}
    \loadFlowVec_{\otherNode} \rev{=0},
    \ \loadFlowVec_{\otherNode} \in [0,2],
    \otherNode \in \exits, \ \loadFlowVec_{0} = 0, \
    \rev{0 \geq }  \loadFlowVec_{\rev{s}} \rev{ \geq -}2 \abs{\exits}}.
\end{equation*}
We add another expansion arc~$(\rev{s},0,\text{large})$
to~$\expArcs$ with \mbox{$\Lambda_{(\rev{s},0,\text{large})} =
  1/(2\abs{\exits}-1)^{2}$}.
This is necessary to guarantee robust feasibility since the maximal
arc flow between~$s$ and~$0$
increases within the new uncertainty set~$\tilde{U}$,
causing a larger potential drop between these nodes.
Applying Algorithm~\ref{alg:adversarial-approach} to the adapted
instance leads to the result that only a single worst-case scenario
is necessary to build a robust network.
This worst-case scenario is given by~$\loadFlowVec_{\rev{s}} = \rev{-}2
\abs{\exits}, \, \loadFlowVec_{0} = 0$, and the remaining exits are at
their maximum, \ie, $\loadFlowVec_{\otherNode} = 2, \otherNode \in
\exits$.
Considering this worst-case scenario leads to the same robust network
as in the original instance except that we \rev{build} the larger arc
between $s$ and~$0$, \ie, we \rev{build} $(\rev{s}, 0 ,\text{large})$
with $\Lambda_{(\rev{s},0,\text{large})} = 1/(2\abs{\exits}-1)^{2}$
instead of~$(\rev{s}, 0, \text{ca})$ with~$\Lambda_{(\rev{s}, 0,
  \text{ca})} = 1$.
Consequently, Algorithm~\ref{alg:adversarial-approach} terminates
after a single iteration with an optimal solution.
More general, for the considered network, it can be shown that the
number of worst-case scenarios is bounded from above
by~$\lceil{2\abs{\exits}}/\abs{\loadFlowVec_{\rev{s}}^{+}}\rceil$,
where $\loadFlowVec_{\rev{s}}^{+}$ is the maximal capacity of the
single source.
This is based on the observation that for robustifying the considered
network, for each sink, it suffices that there is at least one
scenario in which the demand of the sink is at its maximum.
The latter is not necessarily true for general networks.
Concluding, the number of worst-case scenarios necessary to
obtain a robust network directly depends on the capacity of the
single source in the considered network topology.

We emphasize that this topology or related ones with the same
behavior regarding the worst-case scenarios are contained as
subnetworks in
many real-world instances of utility networks.
More precisely, the single source corresponds
to the connection point to a large distribution network and the
remaining network corresponds to a local distribution network.
Thus, it is natural that the single source can provide the demand
of all exits.
For this subnetwork only a single worst-case scenario
is then sufficient to guarantee robust feasibility.
This illustrates that the approach is especially suitable for real-world
utility networks since, most likely, only a few worst-case scenarios
are necessary to \rev{build} a robust network.
This is in line with our computational study, in which we only need a
few worst-case scenarios and this number of scenarios is often
close to the number of different sources in the network.

\section{Computational Study}
\label{sec:computational-study}

We now apply the presented adversarial approach to gas
networks.
To this end, we consider different use cases such as robustifying
existing networks and building new ones from scratch.
In Sections~\ref{sec:gas-instances}
and~\ref{sec:modeling-uncertainty}, we discuss the considered gas
network instances and the uncertainty modeling.
In the Section~\ref{sec:computational-algorithmic-setup}, we then
specify the implementation of
Algorithm~\ref{alg:adversarial-approach} and how we incorporate the
enhanced solution techniques of
Section~\ref{sec:enhanced-solution-techniques}.
Finally, we present and discuss the numerical results in
Section~\ref{sec:numerical-results}.

\subsection{Gas Networks}
\label{sec:gas-instances}

All the instances used in the computational study are based on the
\textsf{GasLib} library; see \textcite{Schmidt_Assmann_2017}.
We consider the network~\textsf{GasLib-40} as
given in~\textcite{Schmidt_Assmann_2017}.
Moreover, we create one larger network based on
the significantly larger instance~\textsf{GasLib-135}.
In continuation of the \textsf{GasLib} notation, we call the newly
created instance~\textsf{GasLib-60}, i.e., it has 60 nodes.
We obtain this network by removing the pipes \textsf{pipe\_46},
\textsf{pipe\_93}, \textsf{pipe\_96}, \textsf{pipe\_103}, \textsf{pipe\_104},
and \textsf{pipe\_106} from the \textsf{GasLib-135} network.
Then, \textsf{GasLib-60} represents the weakly connected component
containing~\textsf{source\_3}.
The resulting network has 60 nodes, 3 sources, 39 sinks, 61 pipes, and 18
compressors.

For the networks $\graph = (\nodes, \arcs)$ of
\textsf{GasLib-40} and~\textsf{GasLib-60}, we now create different
instances as follows.
The set $\nodes$ of nodes remains unchanged.
Since we consider pipe-only networks, we replace all occurring
resistors, control valves, and
compressors by so-called short pipes, i.e., by pipes of zero
length that do not induce any pressure loss ($\Lambda_{\arc} = 0$).
We consider three options on how to choose the existing pipes~$\exArcs$ from
the set~$\arcs$, which consists of pipes and short pipes.

In the option \textsf{unchanged}, we set~$\exArcs=\arcs$, \ie, except
for the replaced active  elements, the set
$\exArcs$ of existing arcs coincides with the set of arcs of the
given \textsf{GasLib} instance.
In the option \textsf{spanning tree}, the existing network is assumed to
be a spanning tree.
Thus, for the given gas network instance, we compute a spanning
tree, which also contains all short pipes of the network.
The arcs of this spanning tree then represent the existing network,
\ie, these arcs are stored in~$\exArcs$.
In the option \textsf{greenfield}, we set~$\exArcs=\emptyset$, \ie, we
assume that no arcs are built yet.

To create the candidate arcs $\expArcs$, we apply the following
procedure.
All short pipes that are already an element of $\exArcs$ are not considered
as possible candidate arcs and the remaining short pipes are added
once to the set of candidate arcs~$\expArcs$.
For each pipe of the original network~$\graph=(\nodes, \arcs)$, we
then add multiple candidate
arcs in parallel with different choices for the corresponding diameter.
To do so, we introduce factors~$\tau_1, \ldots, \tau_m \in (0,\infty)$
with~$m \geq 1$ and then each pipe is added~$m$ times with the
reported diameter of the \textsf{GasLib} multiplied once by $\tau_i$
with $i \in \set{1,\ldots,m}$.
For the instances~\textsf{unchanged} and~\textsf{spanning tree}, we
consider the factors~$\set{0.3,0.7,1.0,1.3}$, \ie, for each
expansion candidate, we have four diameter options.
For the even more challenging~\textsf{greenfield} instances, we choose
three diameter options per expansion candidate given by the
factors~$\set{0.5,1.0,1.5}$.

For estimating the investment costs of building new pipes, we follow
the cost estimation of~\textcite{mischner2015gas2energy,
  Reuss_et_al:2019}.
Consequently, the costs of building a pipe~$\arc$ (in
$\si{\euro\per\meter}{}$, i.e., per length) depend on the
corresponding diameter~$D_{\arc}$ (in \SI{}{\milli\meter}).
These investment costs are computed using $278.24 \exp(1.6 D_{\arc})$.
We further do not charge any costs for building short pipes.

We now brief\/ly discuss the implementation of the
potential-based flow model for the considered case of gas networks.
For an arc~$\arc \in \exArcs \cup \expArcs$, the potential function is
explicitly
given by~$  \potFunc_{\arc}(\massflow_{\arc}) = \Lambda_{\arc}
\massflow_{\arc} \abs{\massflow_{\arc}}$; see \textcite{Gross2019}.
The pressure loss coefficient~$\Lambda_{\arc} \geq 0$ is computed
using the formula
\begin{equation*}
  \Lambda_{\arc} = \left( \frac{4}{\pi} \right)^2 \lambda_{\arc}
  \frac{R_{\text{s}} T_{\text{m}} L_{\arc} z_{\text{m},\arc}}{D_{\arc}^5}
\end{equation*}
with $\lambda_{\arc}$ being the pipe's friction factor given by the formula
of Nikuradse, $R_{\text{s}}$ the specific gas constant, $T_{\text{m}}$
a constant mean temperature, $L_{\arc}$ the pipe's length,
and~$D_{\arc}$ the pipe's diameter.
In addition, $z_{\text{m},\arc}$ is the pipe's mean compressibility factor given by
the formula of Papay and an a priori estimation of the mean pressure.
Furthermore, we set~$\Lambda_{\arc} = 0$ if $\arc$ is a short pipe.
For more information and detailed explanations, we refer
to~\textcite{FuegenschuhGasBook}.

\subsection{Uncertainty Modeling}
\label{sec:modeling-uncertainty}

In the computational study, we consider four different polyhedral
uncertainty sets.
We start with a baseline scenario~\mbox{$\loadFlowVec^{\text{base}}
  \in \reals^{\nodes}$}, which then is
affected by certain load fluctuations.
For \textsf{GasLib-40}, we choose the scenario reported in
\textcite{Schmidt_Assmann_2017} as the baseline scenario.
For the newly created instance
\textsf{GasLib-60},  we
choose the following baseline scenario.
The \load is set to
\SI{-520}{(1000\,N\cubic\metre\per\hour)} for all sources and to
\SI{40}{(1000\,N\cubic\metre\per\hour)} for all sinks.

As introduced in Equation~\eqref{eq:demand-uncertainty-set},
each uncertainty set consists of all balanced \rev{load scenarios}
\begin{equation*}
  L \define \Defset{\loadFlowVec \in \reals^{\nodes}}{\sum_{\node \in
      \entries} \loadFlowVec_{\node} \rev{+} \sum_{\node
      \in \exits}  \loadFlowVec_{\node} \rev{=0},
    \ \loadFlowVec_{\node} = 0, \
    \node \in \innodes}
\end{equation*}
intersected with a non-empty and compact set $Z \subset
\reals^{\nodes}$.
We now discuss our polyhedral choices for this compact set~$Z$.

We start with a simple box as a first uncertainty set, \ie, lower and upper
bounds for the injections and withdrawals are additionally imposed.
For the sinks, we consider a lower bound~$\underline{z}^{-} \in [0,1]$ and
an upper bound~$\bar{z}^{-} \in [1, \infty)$, which indicate the
percentage deviation from the baseline scenario.
Analogously, \rev{and taking into account the negative sign of the
  source loads,} we define a lower
bound~$\underline{z}^{+} \in [0,1]$ and an upper bound~$\bar{z}^{+}
\in [1, \infty)$ \rev{for the sources}.
Then, the box uncertainty set is given by
\begin{equation*}
  U_{\text{box}} \define \Defset{\loadFlowVec \in L}{
    \loadFlowVec_{\node}
    \in [\underline{z}^{-} \loadFlowVec^{\text{base}}_{\node},
    \bar{z}^{-}\loadFlowVec^{\text{base}}_{\node}], \ \node \in \exits,
    \ \loadFlowVec_{\node}
    \in [\rev{\bar{z}^{+}} \loadFlowVec^{\text{base}}_{\node},
    \rev{\underline{z}^{+}}\loadFlowVec^{\text{base}}_{\node}], \ \node \in \entries
  }.
\end{equation*}
For~$\underline{z}^{-} = \underline{z}^{+} = 0$, this uncertainty set
allows that sinks or sources fail completely, \ie, they are
switched off.

Based on this box uncertainty set, we define three further uncertainty
sets.
For the first modification, we ensure by two additional inequalities that
the total amount of injections does not exceed or drop below a certain level
regarding the total injections in the baseline scenario.
Hence, we introduce a percentage bound for the
lower ($\underline{i} \in [0,1]$) and upper ($\bar{i} \in [1,\infty)$)
level of total injections.
Then, the box uncertainty set with additional bounds on the total
injections is given by
\begin{equation*}
  U_{\text{sum}} \define U_{\text{box}}
  \cap \Defset{\loadFlowVec \in L}{\rev{\bar{i}} \sum_{\node \in
      \entries} \loadFlowVec^{\text{base}}_{\node} \leq \sum_{\node \in \entries}
    \loadFlowVec_{\node} \leq \rev{\underline{i}} \sum_{\node \in
      \entries} \loadFlowVec^{\text{base}}_{\node}}.
\end{equation*}

For the next modification of the box uncertainty set~$U_{\text{box}}$, we
bound the absolute difference of deviations from the baseline
scenario for selected pairs of withdrawals.
To do so, we consider a randomly chosen subset of the
sinks~$\tilde{V}_{-} \subset \exits$ and an upper bound for the
absolute difference~$\bar{d} \geq 0$.
Then, for a balanced \loadscenario~$\loadFlowVec \in L$, we add the
inequality
\begin{equation}
  \left| (\loadFlowVec^{\text{base}}_{\node})^{-1} \loadFlowVec_{\node}
    - (\loadFlowVec^{\text{base}}_{\otherNode})^{-1} \loadFlowVec_{\otherNode}
  \right|
  \leq \bar{d}, \quad (\node,\otherNode) \in \tilde{V}_{-} \times
  \tilde{V}_{-}.
  \label{eq:correlation-sinks}
\end{equation}
This leads to the third uncertainty set
\begin{equation*}
  U_{\text{corr}} \define U_{\text{box}} \cap
  \Defset{\loadFlowVec \in L}{\eqref{eq:correlation-sinks} \text{ holds}}.
\end{equation*}

Before we continue with the fourth uncertainty set, let us discuss
some details about Condition~\eqref{eq:correlation-sinks} and its
implementation.
The idea behind this condition is that there could be withdrawals that
follow the same consumption pattern, \eg, due to temperature dependency
in case of an energy carrier used for heating.
In our computational study, we obtain the set~$\tilde{V}_{-}$
in~\eqref{eq:correlation-sinks} by randomly selecting
sinks from~$\exits$ until a certain percentage $w$ of the number of
sinks is reached or just exceeded.
The selected sinks form the set $\tilde{V}_{-}$.

The fourth uncertainty set is the intersection of all three previously
defined uncertainty sets.
Thus, it is given by
\begin{equation*}
  \label{eq:box_sum_corr}
  U_{\text{all}} \define U_{\text{box}}
  \cap U_{\text{sum}}
  \cap U_{\text{corr}}.
\end{equation*}

Table~\ref{tab:uncertainty_sets} provides an overview of the specific
parameterization of the uncertainty
sets used.
We consider the case that the \load of sinks
fluctuates slightly more than the \load of sources, i.e.,
$[\underline{z}^{+}, \bar{z}^{+}] \subset [\underline{z}^{-},
\bar{z}^{-}]$ holds.
Moreover, in case of correlated \rev{loads}, \ie, $U_{\text{corr}}$
and $U_{\text{all}}$, we assume that not all but
\SI{80}{\percent} of the sinks are correlated.

\begin{table}
  \caption{Parameterization of the uncertainty sets.}
  \label{tab:uncertainty_sets}
  \begin{tabular}{cccccccc}
    \toprule
    $\underline{z}^{-}$ & $\bar{z}^{-}$ &
                                          $\underline{z}^{+}$ & $\bar{z}^{+}$ & $\underline{i}$ & $\bar{i}$
    & $\bar{d}$ & $w$ \\
    \midrule
    0.6 & 1.4 & 0.7 & 1.3 & 0.8 & 1.2 & 0.1 & 80 \\
    \bottomrule
  \end{tabular}
\end{table}

\subsection{Algorithmic and Computational Setup}
\label{sec:computational-algorithmic-setup}

We now brief\/ly discuss the implementation of the enhanced solution
techniques of Section~\ref{sec:enhanced-solution-techniques} in
Algorithm~\ref{alg:adversarial-approach}.
More precisely, we consider two different configurations of this
algorithm in the following.
For both approaches, we model the MINLP for the expansion decision
with flow direction variables, \ie, we use
Model~(\ref{eq:expansion-with-flow-directions}).
Analogously, we also model the adversarial
problems~(\ref{eq:maximum-potential-difference}) with flow direction variables,
which we outline in Appendix~\ref{sec:appendix}.
For doing so, we use a single flow direction variable for parallel
pipes because parallel pipes always have the same flow direction.
Further, we require by additional constraints that only a single
expansion pipe can be built in parallel to an existing one.
This is legitimate because building multiple new parallel pipes can be
equivalently reformulated as building a single pipe;
see~\textcite{Lenz2016}.
Based on preliminary computational results, we add the
acyclic inequalities, see Section~\ref{subsec:acyclic-inequalities},
to the occurring MINLPs and also to the convex
relaxations~(\ref{eq:expansion-with-flow-directions-relaxation}).
As explained in Section~\ref{subsec:computing-lower-bounds}, we
further use
the optimal objective value of the expansion decision of the previous
iteration as a lower bound for the objective value of the
MINLPs~(\ref{eq:expansion-with-flow-directions}) and the upcoming
relaxations \textsf{Reduced Relaxation} as well as
Problems~(\ref{eq:expansion-with-flow-directions-relaxation}).
The considered pipe-only gas networks typically do not impose any
bounds on the arc flow because the flow is implicitly bounded by the
potential bounds at the incident nodes.
Consequently, we can dismiss the very large flow bounds of the
\textsf{GasLib} instances and we do not have to solve the adversarial
problems~(\ref{eq:minimum-arc-flow}) and~(\ref{eq:maximum-arc-flow}).
However, prior to each iteration, we apply some basic presolve to
compute tighter lower and upper arc flow bounds by solving
Problems~\eqref{eq:mip-tighten-flow-bounds} in order to strengthen
the formulations.
We note that all instances satisfy the requirements of
Lemma~\ref{lemma:entry-exit-maximum-pot-diff} and, thus, we check
robust feasibility using the characterization of this lemma.
Finally, for both approaches, the scenario set~$S$ only contains the
baseline scenario~$\loadFlowVec^{\text{base}}$ in the first iteration,
\ie, we set~$S = \set{\loadFlowVec^{\text{base}}}$.

After these adaptions, we denote as the baseline approach~(\minlpAcyclic)
the plain version of Algorithm~\ref{alg:adversarial-approach}.
For the second approach~(\convexHeur), we extend this baseline approach by
two methods to compute tighter lower bounds for the optimal objective
value of the robust network design
problem~(\ref{eq:expansion-with-flow-directions}), which are
computed iteratively.
More precisely, before solving the expansion problem, \ie, before
Line~\ref{alg:line-solve-expansion-problem} in
Algorithm~\ref{alg:adversarial-approach}, we first solve the
relaxation~\textsf{Reduced Relaxation} (except for the first
iteration), \ie, we solve the expansion
problem only \wrt\ the last added scenario; see
Section~\ref{subsec:computing-lower-bounds}.
Second, we solve the mixed-integer convex
relaxation~(\ref{eq:expansion-with-flow-directions-relaxation}) \wrt\
$S$, \ie, we consider all found worst-case scenarios.
After solving each of these relaxations, we check if the obtained
solution is feasible for the network design
problem~(\ref{eq:expansion-with-flow-directions}) by solving the
latter with fixed expansion decisions.
If this is the case, we directly go to
Line~\ref{alg:line-begin-check-robust-feasibility} of
Algorithm~\ref{alg:adversarial-approach} and check robust feasibility
of the obtained network design.
If applicable, we also update the best known lower bound that we add
to the upcoming MINLPs or convex relaxations to bound the objective
value from below.
The main intuition behind the approach~\convexHeur consists of
avoiding to solve the challenging
MINLP~(\ref{eq:expansion-with-flow-directions}), whose size increases
from iteration to iteration, by solving an MINLP of ``fixed'' size
(\textsf{Reduced Relaxation}) or a mixed-integer second-order cone
problem~(\ref{eq:expansion-with-flow-directions-relaxation}).

We finally note that the models are implemented in
\textsf{Python~3.7} with \textsf{Pyomo~6.7.0}
(\textcite{Bynum_et_al:2021}) and we solve the
models using~\textsf{Gurobi~10.0.3} (\textcite{gurobi}).
The computations are carried out on a single node of a
server\footnote{\url{https://elwe.rhrk.uni-kl.de/elwetritsch/hardware.shtml}}
with~\textsf{Intel XEON SP 6126} CPUs.
Further, we set a memory limit of \SI{64}{\giga\byte}, a
total time limit of~\SI{24}{\hour}, and limit the number of
\textsf{threads} to 4.
Additionally, we use the \textsf{Python} package \textsf{perprof-py}
(\textcite{SoaresSiqueira2016}) to
produce the performance profiles as described
in~\textcite{Dolan2002}.

\subsection{Numerical Results}
\label{sec:numerical-results}

We now apply the two presented variants~\minlpAcyclic and \convexHeur
of Algorithm~\ref{alg:adversarial-approach} to the gas network
instances described in Section~\ref{sec:gas-instances} and the
uncertainty sets described in Section~\ref{sec:modeling-uncertainty}.
Consequently, for each network, we obtain four different instances.
For most of the instances checking robust
feasibility, \ie, solving the adversarial problems, only has a
moderate contribution to the overall runtimes.
Relative to the runtimes of the algorithm, the total time spent to
solve the adversarial
problems~\eqref{eq:maximal-absolute-loadflow-connectedComponents},
which are MILPs, is below \SI{1.3}{\percent}.
For the more challenging nonlinear adversarial
problems~\eqref{eq:maximum-potential-difference},
the median of the aggregated
runtimes relative to the runtimes of the algorithm is
below~\SI{12.5}{\percent}.
Only in some cases this percentage increases to at
most~\SI{88.26}{\percent}, which is often the case if the algorithm
only needs two or less iterations.
For obtaining these moderate runtimes regarding the
aggregated runtimes of the adversarial problems,
Lemma~\ref{lemma:entry-exit-maximum-pot-diff} is key.
In particular, the approach benefits from the property that utility
networks typically contain a small number
of sources, \eg, the considered instances contain $3$ sources.
Since the runtimes of the adversarial problems are moderate compared
to the total runtimes of the algorithm in most cases, we only focus on the
total runtimes of
Algorithm~\ref{alg:adversarial-approach} in the following.

\subsubsection{Robustifying Existing Networks}

We start by applying our approach to robustify the existing gas
networks~\textsf{GasLib-40} and~\textsf{GasLib-60} by
building new pipes in parallel to existing ones.
In this case, the majority of pipes is already existing and we
selectively expand the capacity of the network to resolve bottlenecks
and to guarantee robust feasibility.
In Table~\ref{table:minlp-acyclic-existing-40-60},
we summarize the statistics
of the total runtimes and the number of added adversarial scenarios
for the plain version of Algorithm~\ref{alg:adversarial-approach},
\ie, for the variant~\minlpAcyclic.
Analogously, we summarize the statistics for the variant~\convexHeur in
Table~\ref{table:conv-heur-acyclic-existing-40-60}.
For these instances at most two adversarial scenarios
are sufficient to compute a robust network.
We emphasize that this small number of adversarial scenarios can be
most likely traced back to the
typical structure of gas networks that a small number of sources can
supply the majority of sinks.
This usually leads to a small number of worst-case scenarios as
illustrated in Section~\ref{sec:worst-case-scenarios}.
Regarding the runtimes, the variant~\convexHeur is slightly faster on
the majority of instances but not on every instance.
This can be explained by the observation that in the
approach~\convexHeur all of the
expansion MINLPs~(\ref{eq:expansion-with-flow-directions}) could be
solved by the relaxations, \ie, by either the \textsf{Reduced
  Relaxation} or the mixed-integer second-order cone
relaxation~(\ref{eq:expansion-with-flow-directions-relaxation}).
We note that this is not the case in general and also does not hold
for the following instances.
Overall, both variants of Algorithm~\ref{alg:adversarial-approach} are
very effective to robustify existing gas networks.

{
  \footnotesize
  \begin{table}
    \caption{Runtimes and number of adversarial scenarios of the
      approach~\minlpAcyclic.
      Left: Instances \wrt\ \mbox{\textsf{unchanged GasLib-40}}.
      Right: Instances \wrt\ \mbox{\textsf{unchanged GasLib-60}}.}
    \centering
    \begin{tabular}{lS[table-format=3.2]S[table-format=4.2]S[table-format=4.2]}
\toprule
\#Solved & {4 of 4} &  &  \\
\midrule
 & {Min} & {Median} & {Max} \\
\midrule
\#Iterations & 2 & 3 & 3 \\
\#Scenarios & 1 & 2 & 2 \\
Runtime (s) & 807.65 & 1395.33 & 1578.68 \\
\bottomrule
\end{tabular}

    \begin{tabular}{lS[table-format=4.2]S[table-format=4.2]S[table-format=4.2]}
\toprule
\#Solved & {4 of 4} &  &  \\
\midrule
 & {Min} & {Median} & {Max} \\
\midrule
\#Iterations & 2 & 2 & 2 \\
\#Scenarios & 1 & 1 & 1 \\
Runtime (s) & 1117.37 & 1175.83 & 3009.57 \\
\bottomrule
\end{tabular}

    \label{table:minlp-acyclic-existing-40-60}
  \end{table}
}
{
  \footnotesize
  \begin{table}
    \centering
    \caption{Runtimes and number of adversarial scenarios of the
      approach~\convexHeur.
      Left: Instances \wrt\ {\textsf{unchanged \mbox{GasLib-40}}}.
      Right: Instances \wrt\ \textsf{unchanged GasLib-60}.}
    \begin{tabular}{lS[table-format=3.2]S[table-format=4.2]S[table-format=4.2]}
\toprule
\#Solved & {4 of 4} &  &  \\
\midrule
 & {Min} & {Median} & {Max} \\
\midrule
\#Iterations & 2 & 3 & 3 \\
\#Scenarios & 1 & 2 & 2 \\
Runtime (s) & 332.21 & 1149.98 & 2042.90 \\
\bottomrule
\end{tabular}

    \begin{tabular}{lS[table-format=3.2]S[table-format=3.2]S[table-format=4.2]}
\toprule
\#Solved & {4 of 4} &  &  \\
\midrule
 & {Min} & {Median} & {Max} \\
\midrule
\#Iterations & 2 & 2 & 2 \\
\#Scenarios & 1 & 1 & 1 \\
Runtime (s) & 564.06 & 995.62 & 1037.74 \\
\bottomrule
\end{tabular}

    \label{table:conv-heur-acyclic-existing-40-60}
  \end{table}
}

\subsubsection{Extending Backbone Networks}

{
  \footnotesize
  \begin{table}
    \centering
    \caption{Runtimes and number of adversarial scenarios of the
      approach~\minlpAcyclic.
      Left: Instances \wrt\ \mbox{\textsf{spanning tree GasLib-40}}.
      Right: Instances \wrt\ \textsf{spanning tree GasLib-60}.}
    \begin{tabular}{lS[table-format=3.2]S[table-format=4.2]S[table-format=4.2]}
\toprule
\#Solved & {4 of 4} &  &  \\
\midrule
 & {Min} & {Median} & {Max} \\
\midrule
\#Iterations & 2 & 2 & 2 \\
\#Scenarios & 1 & 1 & 1 \\
Runtime (s) & 951.95 & 6281.03 & 7818.14 \\
\bottomrule
\end{tabular}

    \begin{tabular}{lS[table-format=3.2]S[table-format=4.2]S[table-format=4.2]}
\toprule
\#Solved & {4 of 4} &  &  \\
\midrule
 & {Min} & {Median} & {Max} \\
\midrule
\#Iterations & 2 & 3 & 3 \\
\#Scenarios & 1 & 2 & 2 \\
Runtime (s) & 274.42 & 1743.79 & 2875.57 \\
\bottomrule
\end{tabular}

    \label{table:minlp-acyclic-spanning-40-60}
  \end{table}
}
{
  \footnotesize
  \begin{table}
    \centering
    \caption{Runtimes and number of adversarial scenarios of the
      approach~\convexHeur.
      Left: Instances \wrt\ \mbox{\textsf{spanning tree GasLib-40}}.
      Right: Instances \wrt\ \textsf{spanning tree GasLib-60}.}
    \begin{tabular}{lS[table-format=3.2]S[table-format=3.2]S[table-format=3.2]}
\toprule
\#Solved & {4 of 4} &  &  \\
\midrule
 & {Min} & {Median} & {Max} \\
\midrule
\#Iterations & 2 & 2 & 2 \\
\#Scenarios & 1 & 1 & 1 \\
Runtime (s) & 312.49 & 576.27 & 726.63 \\
\bottomrule
\end{tabular}

    \begin{tabular}{lS[table-format=3.2]S[table-format=3.2]S[table-format=4.2]}
\toprule
\#Solved & {4 of 4} &  &  \\
\midrule
 & {Min} & {Median} & {Max} \\
\midrule
\#Iterations & 2 & 3 & 3 \\
\#Scenarios & 1 & 2 & 2 \\
Runtime (s) & 215.01 & 805.01 & 1954.97 \\
\bottomrule
\end{tabular}

    \label{table:conv-heur-acyclic-spanning-40-60}
  \end{table}
}

We now consider the case that a spanning tree is given as the existing
network and we expand this network by new pipes, which are not
necessarily in parallel to the existing ones.
We summarize the statistics of the considered approaches in
Tables~\ref{table:minlp-acyclic-spanning-40-60}
and~\ref{table:conv-heur-acyclic-spanning-40-60}.
As before the number of adversarial scenarios is very low, which can be
explained as in the previous section.
Regarding the runtimes, the approach~\convexHeur significantly
outperforms the plain version of
Algorithm~\ref{alg:adversarial-approach}.
For the \textsf{GasLib-40}, this can again be explained by the fact
that all occurring MINLPs could be solved by the relaxations.
This is not the case for all instances of \textsf{GasLib-60}.
However, the majority of the obtained lower bounds for the optimal
objective value of the corresponding MINLP is close to the optimal
value.
More precisely, the mixed-integer second-order cone
relaxation~\eqref{eq:expansion-with-flow-directions-relaxation} solves
the MINLP $5$ out of $7$ times.
Further, the gap\footnote{We compute the gap as proposed by
  \textsf{CPLEX} under
  \url{https://www.ibm.com/docs/en/icos/22.1.1?topic=parameters-relative-mip-gap-tolerance}.}
between the optimal objective value of the relaxation
and the one of the corresponding MINLP is at most~\SI{0.85}{\percent}.
Additionally, the \textsf{Reduced Relaxation} solves $4$ out of $7$
times the MINLP to optimality and the maximal gap
is~\SI{25}{\percent}.
Thus, for the considered instances, solving additional relaxations
significantly speeds up the solution process.
Overall, the approach~\convexHeur is to be preferred to extend
existing networks in a robust way.

\subsubsection{Greenfield Approach}

We finally turn to the greenfield approach in which we design a
network from scratch.
As expected this setup is significantly more challenging than the
previous ones, which is also reflected in the computational results.
As outlined in Table~\ref{table:minlp-acyclic-greenfield-40-60},
the plain version of Algorithm~\ref{alg:adversarial-approach} can only
solve a single instance for
the \textsf{GasLib-40} and a single one for the~\textsf{GasLib-60}
network.
Applying the enhanced variant~\convexHeur significantly improves the
performance.
In particular, it can solve $3$ out of $4$ instances for the
\textsf{GasLib-40} and a single one for the \textsf{GasLib-60}; see
Table~\ref{table:minlp-convex-heur-greenfield-40-60}.
Compared to the previous cases, one can observe that
slightly more adversarial scenarios are necessary to build a robust
network from scratch.
However, the approach still requires only a moderate amount of worst-case
scenarios.
We note that the number of computed adversarial scenarios matches the
number of sources of the network, which is in line with
the explanations provided in
Section~\ref{sec:worst-case-scenarios}.
The improved performance \wrt\ the runtimes of the
approach~\convexHeur can again be explained by tight gaps \wrt\ the objective
values of the relaxations and the corresponding MINLPs.
More precisely, the \textsf{Reduced Relaxation} solves the MINLP $4$
out of $9$ times and the gap is at most~\SI{1.65}{\percent}.
The mixed-integer second-order code
relaxation~\eqref{eq:expansion-with-flow-directions-relaxation} solves
the MINLP $9$ out of $9$ times.
The good performance \wrt\ the gap between the optimal objective value of
the mixed-integer convex
relaxation~\eqref{eq:expansion-with-flow-directions-relaxation} and
the MINLP~\eqref{eq:expansion-with-flow-directions} is in line with
the computational results of~\textcite{Sanches_et_al:2016}.
However, we note that the runtimes for the
relaxation~\eqref{eq:expansion-with-flow-directions-relaxation} and
the MINLP~\eqref{eq:expansion-with-flow-directions}
drastically increase from iteration to iteration.
For the unsolved instances, this results
in reaching the time limit of \SI{24}{\hour}.
In these cases, for the MINLPs and the mixed-integer convex
relaxations~\eqref{eq:expansion-with-flow-directions-relaxation}, it
seems to be the case that proving optimality is
the biggest challenge for the solvers during the solution process.
Thus, both approaches cannot solve all instances within the
set time limit if designing a network from scratch.

Overall, our computational study based on real-world instances reveals
two main insights.
(i) For the considered instances, only a moderate number of
worst-case scenarios is
necessary to compute a robust network design, which makes the
presented approach very effective in practice.
(ii) The variant~\convexHeur significantly outperforms the plain
version of Algorithm~\ref{alg:adversarial-approach}.
Thus, for most of the instances, it is worth solving additional
relaxations to provide lower bounds for the objective value of the
challenging MINLPs, which then speed up the overall solution process.
We finally highlight this effect by the performance
profile in Figure~\ref{fig:performance-profile-runtime}.

{
  \footnotesize
  \begin{table}
    \caption{Runtimes and number of adversarial scenarios of the
      approach~\minlpAcyclic.
      Left: Instances \wrt\ \mbox{\textsf{greenfield GasLib-40}}.
      Right: Instances \wrt\ \textsf{greenfield GasLib-60}.}
    \begin{tabular}{lS[table-format=4.2]S[table-format=4.2]S[table-format=4.2]}
\toprule
\#Solved & {1 of 4} &  &  \\
\midrule
 & {Min} & {Median} & {Max} \\
\midrule
\#Iterations & 2 & 2 & 2 \\
\#Scenarios & 1 & 1 & 1 \\
Runtime (s) & 7320.85 & 7320.85 & 7320.85 \\
\bottomrule
\end{tabular}

    \begin{tabular}{lS[table-format=5.2]S[table-format=5.2]S[table-format=5.2]}
\toprule
\#Solved & {1 of 4} &  &  \\
\midrule
 & {Min} & {Median} & {Max} \\
\midrule
\#Iterations & 3 & 3 & 3 \\
\#Scenarios & 2 & 2 & 2 \\
Runtime (s) & 81895.84 & 81895.84 & 81895.84 \\
\bottomrule
\end{tabular}

    \centering
    \label{table:minlp-acyclic-greenfield-40-60}
  \end{table}
}

{
  \footnotesize
  \begin{table}
    \caption{Runtimes and number of adversarial scenarios of the
      approach~\convexHeur.
      Left: Instances \wrt\ \mbox{\textsf{greenfield GasLib-40}}.
      Right: Instances \wrt\ \textsf{greenfield GasLib-60}.}
    \centering
    \mbox{\begin{tabular}{lS[table-format=4.2]S[table-format=5.2]S[table-format=5.2]}
\toprule
\#Solved & {3 of 4} &  &  \\
\midrule
 & {Min} & {Median} & {Max} \\
\midrule
\#Iterations & 2 & 4 & 4 \\
\#Scenarios & 1 & 3 & 3 \\
Runtime (s) & 4066.79 & 39963.87 & 50183.53 \\
\bottomrule
\end{tabular}

      \begin{tabular}{lS[table-format=5.2]S[table-format=5.2]S[table-format=5.2]}
\toprule
\#Solved & {1 of 4} &  &  \\
\midrule
 & {Min} & {Median} & {Max} \\
\midrule
\#Iterations & 3 & 3 & 3 \\
\#Scenarios & 2 & 2 & 2 \\
Runtime (s) & 51290.35 & 51290.35 & 51290.35 \\
\bottomrule
\end{tabular}
}
    \label{table:minlp-convex-heur-greenfield-40-60}
  \end{table}
}

\begin{figure}
  \centering
  \begin{tikzpicture}
  \begin{semilogxaxis}[legend cell align={left}, const plot,
  cycle list={
  {colorbrewerOrange,solid},
  {colorbrewerPurple,dashed}},
xmin=1, xmax=14.25,
                ymin=-0.003, ymax=1.003,
                ymajorgrids,
                ytick={0,0.2,0.4,0.6,0.8,1.0},
                ylabel={Proportion of Instances},title={},
                legend pos={south east},
                width=0.7\linewidth, %
                height=0.5\linewidth %
            ]
  \addplot+[mark=none, ultra thick] coordinates {
    (1.0000,0.9500)
    (1.2763,0.9500)
    (1.2941,1.0000)
    (14.2452,1.0000)
  };
  \addlegendentry{Reduced\_Convex}
  \addplot+[mark=none, ultra thick] coordinates {
    (1.0000,0.0500)
    (1.0767,0.1000)
    (1.1810,0.1500)
    (1.2134,0.2000)
    (1.2415,0.2500)
    (1.2763,0.3000)
    (1.2941,0.3000)
    (1.2950,0.3500)
    (1.4709,0.4000)
    (1.5967,0.4500)
    (1.8002,0.5000)
    (2.0935,0.5500)
    (2.1662,0.6000)
    (2.4311,0.6500)
    (3.0463,0.7000)
    (3.1890,0.7500)
    (5.3356,0.8000)
    (8.6441,0.8500)
    (13.5668,0.9000)
    (14.2452,0.9000)
  };
  \addlegendentry{MINLP\_Acyclic}
  \end{semilogxaxis}
\end{tikzpicture}
  \caption{Performance profile for the runtimes regarding
    all instances that are solved by at least one approach.}
  \label{fig:performance-profile-runtime}
\end{figure}
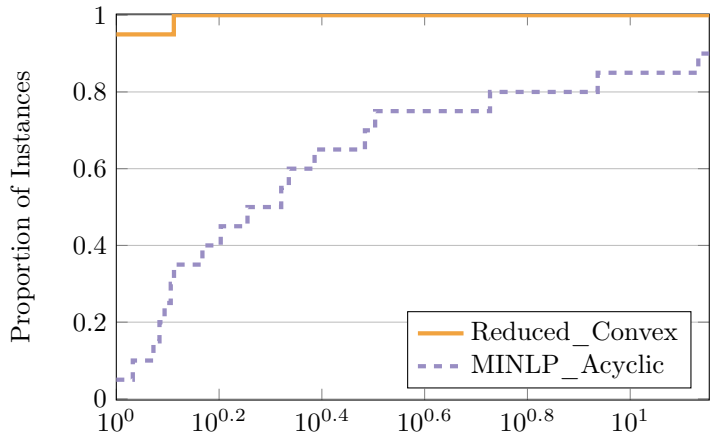

\section{Conclusion}
\label{sec:conclusion}

We studied an adjustable robust mixed-integer nonlinear network design
problem with \loadscenario uncertainties \revTwo{for networks without
  controllable elements}.
To this end, we considered nonlinear potential-based flows, which allow
to model different types of utility networks such as gas, hydrogen, or
water networks.
For the considered problem, we developed an exact adversarial approach
that
exploits the structural properties of the network and flows to obtain
a robust network design that is protected \rev{against} \loadscenario
uncertainties.
Finally, we demonstrated the applicability of the developed approach
using the example of realistic gas networks instances.
The computational results highlight that for these instances only a
very small number of worst-case scenarios is necessary to obtain a
robust network, which makes the presented approach very efficient in
practice.

One promising direction for future research consists of developing
valid inequalities for network expansion problems with nonlinear
potential-based flows.
In contrast to the large literature on valid inequalities for network
design problems with capacitated linear flows, the corresponding
literature on potential-based flows is rather scarce.
\rev{Moreover, including controllable elements such as compressors in
  gas networks or pumps in water networks is a challenging but
  important future research direction
  that leads to solving challenging nonconvex bilevel problems; see
  Remarks~\ref{remark:bilevel-passive-network}
  and~\ref{remark:active-elements}.}
\revTwo{Certain ideas to address controllable elements in specific nonconvex
  bilevel problems can already be found in the literature \parencite{Plein2021}.
  However, these techniques themselves require rather strong
  assumptions such as that the controllable elements are not placed on
  cycle arcs of the network.
  Nevertheless, these ideas might be a good starting point to
  generalize the model discussed in this paper.}

\section*{Acknowledgements}

The authors thank the Deutsche Forschungsgemeinschaft for their
support within project B08 in the
Sonderforschungsbereich/Transregio 154 ``Mathematical Modelling,
Simulation and Optimization using the Example of Gas Networks''.
The computations were executed on the high performance cluster
``Elwetritsch'' at the TU Kaiserslautern, which is part of the
``Alliance of High Performance Computing Rheinland-Pfalz'' (AHRP).
We kindly acknowledge the support of RHRK.

\section*{Declaration}

The authors declare that they do not have any conflict of interest.

\printbibliography

\appendix
\section{Linearization of Bilinear Terms}
\label{sec:appendix}

We now derive an exact reformulation of the bilinear terms in the
left-hand sides of
Constraints~\eqref{cons:potential-link-flow-dir-vars} using McCormick
inequalities.
To do so, for each arc~$\arc=(u,v,\labelArc)$, we linearize the
term~$2y_{\arc}^{\loadFlowVec}(\potential^{\loadFlowVec}_{\node} -
\potential^{\loadFlowVec}_{\otherNode})$
by introducing the additional variable~$\gamma^{\loadFlowVec}_{\arc}
\in \reals$ and the
four inequalities
\begin{subequations} \label{eq:mcCormick}
  \begin{align}
    & 2(\potential^{\loadFlowVec}_{u} - \potential^{\loadFlowVec}_{v})
      - 2(1-y^{\loadFlowVec}_{\arc}) (\ubPot_u-\lbPot_v)
      \leq
      \gamma^{\loadFlowVec}_{\arc},
      \label{mcCormick:one}
    \\
    & 2(\potential^{\loadFlowVec}_{u} - \potential^{\loadFlowVec}_{v})
      - 2(1-y^{\loadFlowVec}_{\arc}) (\lbPot_u-\ubPot_v)
      \geq
      \gamma^{\loadFlowVec}_{\arc},
      \label{mcCormick:two}
    \\
    & 2 y^{\loadFlowVec}_{\arc} (\lbPot_u - \ubPot_v) \leq
      \gamma^{\loadFlowVec}_{\arc},
      \label{mcCormick:three}
    \\
    & 2 y^{\loadFlowVec}_{\arc} (\ubPot_u-\lbPot_{v}) \geq
      \gamma^{\loadFlowVec}_{\arc}.
      \label{mcCormick:four}
  \end{align}
\end{subequations}
If $y^{\loadFlowVec}_{\arc} = 1$ holds, then from
Constraints~\eqref{mcCormick:one} and~\eqref{mcCormick:two}, it
follows $\gamma_{\arc}^{\loadFlowVec} =
2(\potential^{\loadFlowVec}_{u} - \potential^{\loadFlowVec}_{v})$.
Further, every potential vector~$\potential \in \reals^{\nodes}$ that
satisfies the
potential bounds~\eqref{eq:expansion:potential-bounds} also
satisfies together with~$\gamma^{\loadFlowVec}_{\arc}$
the Constraints~\eqref{mcCormick:three}
and~\eqref{mcCormick:four}.
If $y^{\loadFlowVec}_{\arc} = 0$ holds, then from
Constraints~\eqref{mcCormick:three} and~\eqref{mcCormick:four}, it
follows $\gamma_{\arc}^{\loadFlowVec} = 0$.
Further, every potential vector~\mbox{$\potential \in
  \reals^{\nodes}$} that satisfies the
potential bounds~\eqref{eq:expansion:potential-bounds}, then also
satisfies together with~$\gamma^{\loadFlowVec}_{\arc}$
the Constraints~\eqref{mcCormick:one} and~\eqref{mcCormick:two}.
Thus, for every vector $(\potential^{\loadFlowVec}_{u},
\potential^{\loadFlowVec}_{v}, \gamma^{\loadFlowVec}_{\arc})$ that
satisfies~\eqref{eq:mcCormick}, it holds
$\gamma_{\arc}^{\loadFlowVec}=2y_{\arc}^{\loadFlowVec}(\potential^{\loadFlowVec}_{\node}
- \potential^{\loadFlowVec}_{\otherNode})$.

Using the previous linearization, we can replace the bilinear terms in
the left-hand sides of \eqref{cons:potential-link-flow-dir-vars}
by~$\gamma^{\loadFlowVec}_{\arc}$.
Consequently, Constraints~\eqref{eq:mcCormick} together with the
constraints
\begin{subequations} \label{cons:linearization-potential-link}
  \begin{align}
    & (\potential^{\loadFlowVec}_{\otherNode} -
      \potential^{\loadFlowVec}_{\node})
      +
      \gamma^{\loadFlowVec}_{\arc}
      =
      \potFunc_{\arc}(\abs{\massflow^{\loadFlowVec}_{\arc}}), \quad
      \arc=(\node,\otherNode, \labelArc)
      \in \exArcs,
    \\
    &
      (\potential^{\loadFlowVec}_{\otherNode} -
      \potential^{\loadFlowVec}_{\node})
      +
      \gamma^{\loadFlowVec}_{\arc}
      \leq
      \potFunc_{\arc}(\abs{\massflow^{\loadFlowVec}_{\arc}}) +
      (1-\expVar_{\arc}) M^{+}_{\arc},
      \quad \arc=(\node,\otherNode, \labelArc)
      \in \expArcs,
    \\
    &
      (\potential^{\loadFlowVec}_{\otherNode} -
      \potential^{\loadFlowVec}_{\node})
      +
      \gamma^{\loadFlowVec}_{\arc}
      \geq
      \potFunc_{\arc}(\abs{\massflow^{\loadFlowVec}_{\arc}}) +
      (1-\expVar_{\arc}) M^{-}_{\arc},
      \quad \arc=(\node,\otherNode, \labelArc)
      \in \expArcs,
  \end{align}
\end{subequations}
leads to an equivalent reformulation of the
Constraints~\eqref{cons:potential-link-flow-dir-vars}.

Analogously, we can linearize the bilinear terms in the
Relaxation~\eqref{eq:expansion-with-flow-directions-relaxation}.
This then leads to a mixed-integer second-order cone problem for the
considered case of gas networks, in which the potential
function satisfies~$\potFunc_{\arc}(\abs{\massflow_{\arc}}) =  \Lambda_{\arc}
\massflow_{\arc}^{2}$.

We finally discuss that we can use the previous
linearization~\eqref{eq:mcCormick}
and~\eqref{cons:linearization-potential-link} with minor adaptions to
model the adversarial
problems~\eqref{eq:maximum-potential-difference}--%
\eqref{eq:maximum-arc-flow} as MINLPs.
Since the adversarial problems do not contain lower or upper potential
bounds, we have to replace these potential bounds in
Constraints~\eqref{mcCormick:one}--\eqref{mcCormick:four}
by valid bounds so that the optimal value
of an optimal solution to the adversarial
problems~\eqref{eq:maximum-potential-difference}--\eqref{eq:maximum-arc-flow}
is not changed.
To this end, for each arc~$\arc \in \arcs$, we can compute a lower and
upper arc flow bound \wrt\ the uncertainty set by
solving the optimization problems
\begin{equation*}
  \max_{\massflow^{\loadFlowVec}, \loadFlowVec}
  \quad \massflow^{\loadFlowVec}_{\arc} \quad
  \st \quad \eqref{eq:expansion:massflow-conservation}, \
  \eqref{eq:acyclic-inequalities},
  \, \loadFlowVec \in \uncertaintySet,
  \quad
  \min_{\massflow^{\loadFlowVec}, \loadFlowVec}
  \quad \massflow^{\loadFlowVec}_{\arc} \quad
  \st \quad \eqref{eq:expansion:massflow-conservation}, \
  \eqref{eq:acyclic-inequalities},
  \, \loadFlowVec \in \uncertaintySet.
\end{equation*}
Similar to Problem~\eqref{eq:mip-tighten-flow-bounds}, these problems
compute upper and lower arc flow bounds by solving an uncapacitated
linear flow model with the additional restriction of acyclic flows
over the given uncertainty set \rev{of load scenarios}.
We now denote a corresponding bound on the maximum
absolute arc flow of~$\arc$ by~$\tilde{\massflow}_{\arc}$.
From this, we obtain the inequalities
$-\Lambda_{\arc} \tilde{\massflow}_{\arc}^{2}
\leq
\potential^{\loadFlowVec}_{u} - \potential^{\loadFlowVec}_{v}
\leq
\Lambda_{\arc} \tilde{\massflow}_{\arc}^{2}
$,
which are valid for all feasible points of the adversarial
problems~\eqref{eq:maximum-potential-difference}--\eqref{eq:maximum-arc-flow}.
Consequently, for applying the McCormick inequalities and model the adversarial
problems~\eqref{eq:maximum-potential-difference}--\eqref{eq:maximum-arc-flow}
as MINLPs, we replace in the above linearization the
Constraints~\eqref{mcCormick:one}--\eqref{mcCormick:four} by
\begin{align*}
  2(\potential^{\loadFlowVec}_{u} - \potential^{\loadFlowVec}_{v})
  - 2(1-y^{\loadFlowVec}_{\arc}) \Lambda_{\arc} \tilde{\massflow}_{\arc}^{2}
  \leq
  \gamma^{\loadFlowVec}_{\arc},
  & \quad
    2(\potential^{\loadFlowVec}_{u} - \potential^{\loadFlowVec}_{v})
    + 2(1-y^{\loadFlowVec}_{\arc}) \Lambda_{\arc} \tilde{\massflow}_{\arc}^{2}
    \geq
    \gamma^{\loadFlowVec}_{\arc},\\
  \quad
  -2 y^{\loadFlowVec}_{\arc} \Lambda_{\arc} \tilde{\massflow}_{\arc}^{2} \leq
  \gamma^{\loadFlowVec}_{\arc},
  & \quad
    2 y^{\loadFlowVec}_{\arc} \Lambda_{\arc} \tilde{\massflow}_{\arc}^{2} \geq
    \gamma^{\loadFlowVec}_{\arc}.
\end{align*}
We note that for this linearization, we can use any
bound~$\tilde{\massflow}_{\arc}$ on the absolute flow.
Finally, we can fix
the potential level of an arbitrary node in the adversarial
problems~\eqref{eq:maximum-potential-difference}--\eqref{eq:maximum-arc-flow}
due to Lemma~\ref{lemma:uniquness-of-flows-passive-case}.

\end{document}